\documentclass[12pt]{amsart}

\makeatletter
\usepackage{mathtools}
\usepackage{amsmath, amssymb, amsfonts, amstext, amsthm}
\usepackage{caption,subcaption,float,thmtools}
\usepackage{tikz}
\usetikzlibrary{arrows,positioning,quotes, calc}
\usepackage[colorlinks=true,linkcolor=blue,citecolor=blue,urlcolor=blue,citebordercolor={0 0 1},urlbordercolor={0 0 1},linkbordercolor={0 0 1}]{hyperref} 
\usepackage{cleveref}
\tikzset{
  strata/.style={
    >=latex,
    every node/.style={font=\small,inner sep=1.2pt,align=center},
    every label/.style={font=\small},
  },
  vertex/.style={
    draw,
    circle,
    minimum size=6mm,  
    inner sep=1.5pt,
    font=\small,
  },
  mark/.style={
    circle,
    inner sep=1.5pt,
    fill=black,
  },
  edgelabel/.style={
    midway,
    fill=white,
    inner sep=1pt,
    font=\small,  
  },
  leglabel/.style={
    font=\small,         
  },
}

\def\makeCal#1{\expandafter\newcommand\csname c#1\endcsname{\mathcal{#1}}}
\def\makeBB#1{\expandafter\newcommand\csname b#1\endcsname{\mathbb{#1}}}
\def\makeFrak#1{\expandafter\newcommand\csname f#1\endcsname{\mathfrak{#1}}}
\def\makeRM#1{\expandafter\newcommand\csname r#1\endcsname{\mathrm{#1}}}

\count@=0
\loop
\advance\count@ 1
\edef\y{\@Alph\count@}\expandafter\makeCal\y
\expandafter\makeBB\y
\expandafter\makeFrak\y
\expandafter\makeRM\y
\ifnum\count@<26
\repeat

\theoremstyle{plain}
\newtheorem{thm}{Theorem}[section]
\newtheorem{cor}[thm]{Corollary}
\newtheorem{lem}[thm]{Lemma}

\newtheorem{claim}[thm]{Claim}

\newtheorem{prop}[thm]{Proposition}

\theoremstyle{definition}
\newtheorem{rem}[thm]{Remark}
\newtheorem{defn}[thm]{Definition}

\newtheorem{innercustomthm}{Proposition}
\newenvironment{customprop}[1]
  {\renewcommand\theinnercustomthm{#1}\innercustomthm}
  {\endinnercustomthm}

\DeclareMathOperator{\Eff}{Eff}

\DeclareMathOperator{\Aut}{Aut}

\DeclareMathOperator{\im}{im}
\DeclareMathOperator{\coker}{coker}

\newcommand\ocH{{\overline\cH}}
\newcommand\cHgnm{\cH_{g,n}(\mu)}
\newcommand\ocHgnm{\overline\cH_{g,n}(\mu)}

\newcommand\LG{\mathrm{LG}}
\newcommand\Gr{\mathrm{Gr}}
\newcommand\CH{CH}
\newcommand\RH{RH}
\newcommand\cohH{H}

\makeatletter
\let\@wraptoccontribs\wraptoccontribs
\makeatother

\author{Prabhat Devkota}
\address{Department of Mathematics, Stony Brook University, Stony Brook, NY 11794-3651}
\email{prabhat.devkota@stonybrook.edu}

\author{Samuel Grushevsky}
\address{Department of Mathematics and Simons Center for Geometry and Physics, Stony Brook University, Stony Brook, NY 11794-3651}
\email{sam@math.stonybrook.edu}
\thanks{Research of both authors is supported in part by NSF grant DMS-21-01631.}

\contrib[with an appendix joint with]{Dawei Chen and Martin Möller}
\thanks{Research of Dawei Chen is supported in part by NSF grant DMS-23-01030 and by Simons Travel Support for Mathematicians.}
\address{Department of Mathematics, Boston College, Chestnut Hill, MA 02467, USA}
\email{dawei.chen@bc.edu}

\address{Institut f\"ur Mathematik, Goethe-Universit\"at Frankfurt,
Robert-Mayer-Str. 6-8,
60325 Frankfurt am Main, Germany}
\email{moeller@math.uni-frankfurt.de}
\thanks{Research of Martin M\"oller is supported by the DFG-project MO 1884/2-1 and by the Collaborative Research Centre TRR 326 ``Geometry and Arithmetic of Uniformized Structures''.}

\begin{document}

\title[Non-trivial cohomology of strata of differentials]{Non-trivial and non-tautological cohomology of strata of differentials}
\begin{abstract}
  In this paper we construct various non-trivial and non-tautological cohomology classes on compactified and uncompactified strata of curves with a differential, by using the geometry of the boundary stratification of the moduli space of multi-scale differentials.
\end{abstract}
\maketitle
\section{Introduction}
For $g\ge 0$ and a set of integers $m_1\ge \dots\ge m_n$ summing to~$2g-2$, the  stratum $\cH_{g,n}(\mu)$ is the moduli space of complex curves with marked points $(X,z_1,\dots,z_n)\in\cM_{g,n}$ together with a non-zero meromorphic differential~$\omega$ on~$X$ such that its divisor of zeroes and poles is $\operatorname{div}(\omega)=\sum m_i z_i$. We work with the projectivized stratum, that is we consider~$\omega$ up to scaling by a non-zero complex number, and the points~$z_i$ are numbered and distinct. While one can consider~$\cHgnm$ as a subvariety of~$\cM_{g,n}$, it is often more convenient to consider it abstractly, with various forgetful maps.

In \cite{bcggm} a compactification~$\ocHgnm$ of the strata, called the moduli space of multi-scale differentials, was constructed. It is a projective smooth Deligne-Mumford stack (see \cite{chcomo},\cite{cghms}) with a normal crossing boundary, which is simple normal crossing outside of the horizontal boundary divisor (see \cite{comoza}). The boundary of~$\ocHgnm$ is stratified, with strata being (up to finite correspondences described in detail in \cite{comoza} and reviewed below as necessary) themselves generalized strata of differentials, that is strata of differentials on possibly disconnected curves, satisfying a number of relations of the form that sums of certain sets of residues are equal to zero.

This stratification allows us to obtain some results on the cohomology of the strata of differentials by utilizing and developing techniques recently successfully applied for studying the cohomology of~$\overline\cM_{g,n}$. While our main focus in this paper are the strata of holomorphic differentials, the constructions use the rich geometry of the strata of (generalized) meromorphic differentials that appear in the boundaries of their multi-scale compactifications.

Little is known about the cohomology of strata of differentials beyond the few examples of holomorphic strata in low genus. 
Motivated partially by the recent progress on cohomology of moduli of curves, including \cite{befapa}, \cite{calapawi}, and in particular some ideas from \cite{calapa}, we mostly investigate higher degree cohomology classes on strata of differentials. Our main results are as follows. 

\begin{thm}\label{thm:H3}
For any stratum in genus~$g\ge 3$ (and for $g=2$ provided $n\ge 3$) such that~$m_1\ge 9$, the third cohomology is non-zero: $\cohH^3(\ocHgnm)\ne 0$.

Moreover, for~$n$ and $m_2,\dots,m_n$ fixed, $\dim \cohH^3(\ocHgnm)$ grows at least exponentially in~$m_1$ (or equivalently, in~$g$).
\end{thm}

We remark that the above theorem will be  complemented by an ongoing systematic investigation on the first cohomology of strata of differentials in \cite{CGMP-first}.

Tautological classes on strata of differentials were defined in \cite{comoza}. We will use the biggest tautological ring defined there, i.e.~the smallest system of subrings $\RH^*(\ocHgnm)\subset \cohH^*(\ocHgnm)$ for all~$\mu$ that contains the fundamental classes of each irreducible boundary stratum (not just the unions of those, with a fixed enhanced level graph), and is closed under forgetting marked points of order~$m_i=0$ and under the map from generalized strata to the boundary of bigger strata (see \Cref{df:taut} for details). In particular tautological cohomology classes are all of even degree, and we will construct many even degree non-tautological cohomology classes.

\begin{thm}\label{thm:H4}
For any stratum in genus~$g\ge 5$ (and for $g=4$ provided $n\ge 3$) such that~$m_2\ge 9$, there exist non-tautological algebraic degree~$4$ cohomology classes, i.e.~$\RH^4(\ocHgnm)\subsetneq \cohH^{2,2}(\ocHgnm)$.  

Moreover, for~$n$ and $m_2\ge 9,m_3,\dots,m_n$ fixed, the dimension of the space of such classes grows at least exponentially in~$m_1$.
\end{thm}

\begin{thm}\label{thm:H6}
For any stratum in genus~$g\ge 5$ (and for $g=4$ provided $n\ge 3$) such that $m_1\ge 13$, there exist non-tautological algebraic degree 6 cohomology classes, i.e.~$\RH^6(\ocHgnm)\subsetneq \cohH^{3,3}(\ocHgnm)$. 

Moreover, for~$n$ and $m_2,\dots,m_n$ fixed, the dimension of the space of degree 6 non-tautological cohomology classes grows at least exponentially with~$m_1$.
\end{thm}
For all of the theorems above, keep in mind that we always order the~$m_i$ in a non-increasing way: $m_1\ge m_2\ge\dots \ge m_n$.
\begin{rem}
Recall that the strata of differentials are not necessarily connected, though most are. The connected components of holomorphic strata were classified in \cite{kozo}, and of the meromorphic strata --- in \cite{boissy}. Our theorems above apply to the union of all connected components of a given stratum, i.e.~we claim that {\em some} connected component of a stratum has non-trivial or non-tautological cohomology. We expect that the results actually hold for all non-hyperelliptic components of the strata, possibly with higher bounds for~$m_1$, but determining this would require redoing \cite{leetahar} for the individual connected components of genus~$1$ residueless strata, and investigations as to which boundary strata lie in which connected component, and we do not pursue this here.
\end{rem}

\begin{rem}
Our results in particular construct non-trivial and non-tautological cohomology for most holomorphic strata, but there remain infinite sequences of holomorphic strata, with all $m_i\le 8$, for which our results do not apply. We do not claim that our results above are optimal, in that it is possible that the respective (non-tautological) cohomology is also non-zero for other strata not satisfying the numerical conditions of the theorems above. Furthermore, our methods also yield exponential growth for the dimensions of cohomology in various other sequences of strata with~$m_1$ and/or~$g$ increasing.

We note that the existence of non-trivial odd degree cohomology implies that the Chow ring $\CH^*(\ocH)$ is uncountable, and we prove this for most holomorphic strata in genera $g\geq 6$, including the cases when all $m_i\le 8$ --- see \Cref{cor:nontrivialH5} and \Cref{cor:uncountable_chow}.
\end{rem}

Our constructions of the above cohomology classes start from the existence of non-trivial first cohomology of most one-dimensional genus~$1$ strata of residueless meromorphic differentials. These generalize the well-known cases of level modular curves, whose genus goes to infinity as the level increases. Given such non-zero classes in~$\cohH^1$ of a residueless genus~$1$ stratum, we realize them on a suitable boundary stratum of~$\ocHgnm$, including for the strictly holomorphic strata, and use this to construct higher degree cohomology classes on~$\ocHgnm$ under inclusion. Clearly this method does not yield any non-trivial cohomology classes on the open strata~$\cHgnm$. Nonetheless, we use the weight spectral sequence to construct non-trivial cohomology classes on open strata in a few cases.

\begin{prop}\label{prop:H2mer}
Let~$\cH$ be the connected component of $\cH_{1,n}(\mu)$ of rotation number~$1$. If $n\ge 3$, and either $m_k=\pm 1$ for some~$k$ and $\mu\not\in \{(1,1,-2),(2,-1,-1)\}$, or $\gcd(m_1,\ldots,m_n)>1$, then the second cohomology~$\cohH^2(\cH)$ contains non-trivial classes of weight~$4$ in its mixed Hodge structure (which in particular do not extend to~$\ocH$), 

Furthermore, $\cohH^2(\cH_{1,3}(a+1,-1,-a))$ contains non-trivial classes of weight 3 for $a=10$ or for $a\ge 12$.
\end{prop}
This in particular provides first examples of open strata~$\cH$ with non-trivial second cohomology classes that do not extend to~$\ocH$, in particular with~$\cohH^2$ not generated by the tautological classes~$\lambda$, and~$\psi_i$ associated to simple poles (that is, with~$m_i=-1)$. Note that in \cite[Ex.~2.5]{chentautological} Dawei Chen shows that $\cohH^2(\cH_{g,2g+2}(1,\dots,1,-1,-1))$ contains two linearly independent~$\psi$ classes.

However, we show that one graded piece of second cohomology vanishes for all holomorphic strata.
\begin{prop}\label{prop:H2hol}
  For any strictly holomorphic open stratum~$\cH$, the graded piece $\Gr_4^W\,\cohH^2(\cH)$ is zero.
\end{prop}
The proof of this proposition will use the following statement that was in fact known to the experts, but whose proof is missing in the available literature. In the appendix, joint with Dawei Chen and Martin M\"oller, we give a proof of

\begin{customprop}{A.1}\label{prop:indep}
The irreducible components of the boundary of the multi-scale compactification~$\ocHgnm$ of any holomorphic stratum are linearly independent in $\cohH^2(\ocHgnm)$.
\end{customprop}
We note that this statement trivially fails for meromorphic strata, as already for $(g,n)=(0,4)$ all boundary points of~$\overline\cM_{0,4}$ are equal in cohomology. This independence also fails for higher genus meromorphic strata (see \Cref{section: open_strata} for details).

Furthermore, the proof of \Cref{prop:indep} also yields the extremality of the vertical boundary divisors of holomorphic strata:
\begin{customprop}{A.2}\label{prop:extreme}
Every irreducible component of the vertical boundary of every connected component~$\ocH$ of every holomorphic stratum~$\ocHgnm$ is extremal in the cone of pseudo-effective divisors $\overline{\Eff}^1(\ocH)$.
\end{customprop}
The proof of this proposition will also be given in the appendix.

\subsection*{Structure of the paper.} In \Cref{section:background}, we recall the basic notions on the moduli spaces of multi-scale differentials constructed in \cite{bcggm} and summarize the results from \cite{chcomo, comoza, leetahar} that will be useful throughout the paper. In \Cref{section:nontrivial_H3}, we construct non-trivial cohomology classes in odd degrees and prove \Cref{thm:H3}. In \Cref{section:nontautological_even}, we construct non-tautological even degree cohomology classes and prove \Cref{thm:H4} and \Cref{thm:H6}. In \Cref{section: open_strata}, we switch our focus to the cohomology of the interior. In particular, we recall the mixed Hodge structure on the cohomology of the open strata of differentials described by the weight spectral sequence and use it to prove \Cref{prop:H2mer} and \Cref{prop:H2hol}. In \Cref{proof_prop:genusLT}, we closely examine the explicit expressions for Euler characteristics of the genus~$1$ residueless strata from \cite{leetahar} and prove \Cref{prop:genusLT} that characterizes the signatures for which all the components of the residueless strata are rational. And finally, in \Cref{lin_indep}, written jointly with Dawei Chen and Martin M\"oller, we prove \Cref{prop:indep}, showing the linear independence of the irreducible components of the boundary divisors of holomorphic strata in cohomology, and  \Cref{prop:extreme}, showing the extremality of their vertical boundary divisors in the cone of pseudo-effective divisors.

\subsection*{Acknowledgements} 
We are grateful to Dawei Chen and Martin M\"oller for inspiring discussions around cohomology of strata, and for telling us about the results of their ongoing work on the first cohomology of the strata. We are grateful to Myeongjae Lee and Guillaume Tahar for discussions and explanations regarding the intricacies of residueless genus~$1$ strata. Dawei Chen and Martin M\"oller would like to thank Matteo Costantini and Jonathan Zachhuber for enlightening discussions related to the ideas utilized in the appendix. 

\section{Notation and Background}\label{section:background}
Here we briefly recall the notation for multi-scale differentials and summarize the results from \cite{bcggm,comoza,chcomo} that we will need. We will write~$\cH$ for a fixed connected component of a given stratum~$\cHgnm$. All (co)homology is always considered with rational coefficients.

\subsection*{The multi-scale compactification}
For every boundary stratum of~$\ocH$ there is a corresponding enhanced level graph~$\Gamma^+$, which is a dual graph~$\Gamma$ of a stable curve $X=\cup_{v\in V(\Gamma)} X_v$ with legs of the graph being the zeroes and poles~$z_i$ of the differential, together with a {\em level function} $\ell:V(\Gamma)\twoheadrightarrow \{0,\dots,-L\}$ on the vertices, and with {\em enhancements} $\kappa_e\in \bZ_{\ge 0}$ on each edge~$e\in E(\Gamma)$. On each component of the normalization~$\widetilde X_v$ there is a non-zero meromorphic differential~$\eta_v$ with prescribed zeroes and poles at those marked points~$z_i$ that are contained in~$X_v$, and with a zero (resp.~pole) of order~$\kappa_e-1$ (resp.~$-\kappa_e-1$) at the preimage of the node~$e$ if the corresponding edge goes down (resp.~up) from~$v$. In particular the edges with $\kappa_e=0$ must connect vertices of the same level, where the multi-scale differential must have a simple pole. We will also speak of {\em level passages} $1,\dots, L$ of an enhanced level graph, where $i$'th level passage is the imaginary line between levels~$1-i$ and~$-i$.

In addition, multi-scale differentials must satisfy the global residue condition, which we will review explicitly in the cases we need. To define a boundary stratum of~$\ocH$, one additionally prescribes a prong-matching at every vertical node, i.e.~an order-reversing isomorphism of $\bZ/\kappa_e\bZ$ to itself, where these two copies of $\bZ/\kappa_e \bZ$ are the choices of tangent directions at the two preimages of the node, in local coordinates in which~$\eta$ takes the standard form. 

A multi-scale differential is a tuple $(X,\textbf{z},\Gamma^+,\eta,\sigma)$, where $(X,\textbf{z})\in \overline{\cM}_{g,n}$ is an $n$-pointed stable curve of genus~$g$, $\Gamma^+$~is an enhanced level graph compatible with the dual graph of $(X,\textbf{z})$, $\eta=\{\eta_v\}_{v\in V(\Gamma^+)}$ is a collection of differentials with prescribed zeros and poles on each component~$X_v$ of~$X$ and~$\sigma$ is a choice of prong-matching on each vertical edge of~$\Gamma$. For brevity, we will frequently omit parts of the data and write simply $(X,\textbf{z},\eta)$ or $(X,\eta)$ or $(X,\textbf{z})$. 

The moduli space of multi-scale differentials~$\ocHgnm$ parameterizes equivalence classes of multi-scale differentials, up to the action of {\em level rotation torus}. This is a torus (isogenous to~$(\bC^*)^L$) that acts by simultaneously rescaling the differentials of the same level, and by rotating the prong-matchings.

Given an enhanced level graph~$\Gamma=\Gamma^+$ with enhancements $\{\kappa_e\}_{e\in E(\Gamma)}$, there are~$\prod_e\kappa_e$ choices of prong-matchings on a compatible multi-scale differential. The level rotation group $R_\Gamma\cong\bZ^{L(\Gamma)}\subset \bC^{L(\Gamma)}$ acts by rotation on the prong-matchings, with the level~$i$ generator adding one to prong-matchings on all edges that intersect the $i$'th level passage. The stabilizer of all prong-matchings simultaneously is the {\em twist group} $Tw_\Gamma\subset R_\Gamma$. Not all elements of~$Tw_\Gamma$ can be written as a product of prong rotations that act on one level passage only. Those that can be written in that way form the {\em simple twist group} $Tw_\Gamma^s\subset Tw_\Gamma$. The quotient $K_\Gamma\coloneqq Tw_\Gamma/Tw_\Gamma^s$ is called the group of {\em ghost automorphisms} of a multi-scale differential. As explained in \cite[\S~2.2]{chcomo}, the isotropy group of a multi-scale differential $(X,\eta)$ is an extension of the group of automorphisms~$\Aut(X,\eta)$ by~$K_\Gamma$.

Recall that the codimension in~$\ocH$ of any irreducible component of the boundary stratum with enhanced level graph~$\Gamma^+$ is equal to~$H+L$, where~$L$ is the number of levels below zero, and~$H$ is the number of horizontal edges. In particular, boundary divisors of~$\ocH$ occur for the case $H=1$, $L=0$, called {\em horizontal boundary}, and for the case $H=0$, $L=1$, called {\em vertical} boundary. Following \cite{comoza}, we denote by $\LG_{L,H}$ the set of enhanced level graphs with~$L+1$ levels and~$H$ horizontal edges. To keep notation manageable, we will mostly simply write~$\Gamma$ for enhanced level graphs from now on. Additionally, unless otherwise stated, we will denote by~$D_\Gamma$ a component of the boundary stratum represented by the level graph~$\Gamma$.

Recall from \cite{comoza} that away from the horizontal boundary divisors, the boundary of~$\ocH$ is {\em simple} normal crossing. In fact, for any $\Gamma\in \LG_{L,0}$, for any irreducible component~$D_\Gamma$ of the boundary stratum associated to the level graph~$\Gamma$, let~$D_i$ be the irreducible vertical boundary divisor obtained by smoothing all but $i$'th level passage in~$\Gamma$. Then~$D_\Gamma$ is an irreducible component of $D_1\cap\dots\cap D_L$, and whenever~$D_i$ and~$D_j$ intersect anywhere else within~$\ocH$, for $i<j$,~$D_i$ will still correspond to a degeneration at a higher level than~$D_j$. Furthermore,~$D_\Gamma$ does not intersect any other irreducible component of the stratum with the same enhanced level graph.

\subsection*{Residueless strata}
For any level $i\in \{0,\dots,-L\}$, denote by $\Gamma^{(i)}\subset\Gamma$ the subgraph consisting of all vertices~$v$ at level~$i$, i.e.~with $\ell(v)=i$, together with all (necessarily horizontal) edges connecting them. Then the differentials~$\eta_v$ on all~$\widetilde X_v$ for all $v\in V(\Gamma^{(i)})$ are meromorphic differentials with prescribed orders of zeroes and poles at those~$z_j$ and those preimages of the nodes that lie in~$X_v$, satisfying a number of linear conditions on their residues (opposite residues at the two preimages of a horizontal node, and global residue conditions). Such a subset of a stratum of differentials is called a {\em generalized} stratum of differentials. Within a given stratum~$\cHgnm$ without simple nodes (so all $m_i\ne-1$), the smallest possible generalized stratum is the {\em residueless stratum} $\cR_{g,n}(\mu)\subset\cHgnm$. The connected components of residueless strata were determined by Lee \cite{lee}, while the connected components of generalized strata on connected Riemann surfaces in full generality were determined by Lee and Wong \cite{leewong}. 

We will rely upon the work of Lee and Tahar \cite{leetahar} who determined the orbifold and manifold Euler characteristics of the genus~$1$ residueless strata $\cR_{1,1+n}(a,-b_1,\dots,-b_n)$ (which are complex curves). In particular by directly examining their formula, we prove the following
\begin{prop}\label{prop:genusLT}
For $b_1,\dots,b_n\ge 2$ and $a=b_1+\dots+b_n$, the genus~$1$ residueless stratum $\cR_{1,1+n}(a,-b_1,\dots,-b_n)$ is a possibly disconnected curve all of whose components are rational if and only if it is
$$\cR_{1,2}(a,-a)=\cH_{1,2}(a,-a)\quad \hbox{for\ } a\le 10\ \hbox{or\ } a=12,$$ 
or $(a,-b_1,\dots,-b_n)$ is one of the following: 
\begin{equation}\label{eq:exceptionsLT}
\begin{aligned}
(4,-2,-2),\ \ \ (5,-2,-3),\\ (6,-2,-4),\ (6,-3,-3),\ (6,-2,-2,-2),\\
(7,-2,-5),\ (7,-3,-4),\ (7,-2,-2,-3),\\
(8,-2,-6),\ (8,-3,-5),\ (8,-4,-4),\ (8,-2,-2,-4),\\
(8,-2,-3,-3),\ (8,-2,-2,-2,-2),\ \ (10,-2,-2,-2,-2,-2).
\end{aligned}
\end{equation}
\end{prop}
By inspection, we note that the list above shows that for $a\le 8$ all connected components of all one-dimensional residueless genus~$1$ strata with a unique zero of this order~$a$ are rational curves, while for any $a\ge 9$ almost all (with one exception for $a=9$ and $a=12$ and two exceptions for $a=10$) residueless strata have a connected component that is a curve of higher genus. This explains the numerics in our constructions of non-zero and non-tautological cohomology classes.

We will prove \Cref{prop:genusLT} in \Cref{proof_prop:genusLT} by a numerical argument bounding the explicit expressions for Euler characteristics of the strata obtained in \cite{leetahar}. 

\subsection*{Generalized (product) strata}
Given an enhanced level graph~$\Gamma^+$, for each level~$i$, consider the product $B_\Gamma^{[i]}$ of the {\em unprojectivized} generalized strata of differentials at all vertices $v\in V(\Gamma^{(i)})$, subject to all the residue conditions, and up to the action of~$\bC^*$ that scales all differentials on these components at once. If $i=0$ or~$-L$, we will also use the notation $B_\Gamma^\top$ or $B_\Gamma^\bot$, respectively, to denote $B_\Gamma^{[i]}$. Denoting $B_\Gamma\coloneqq \prod_{i=0}^{-L} B_\Gamma^{[i]}$, \cite[Prop.~4.4]{comoza} provides for each irreducible component~$D_\Gamma$ of the boundary stratum of~$\ocH$ with enhanced dual graph~$\Gamma^+$ the existence of a proper stack~$D_\Gamma^s$ (called the simple boundary stratum), and two finite morphisms
\begin{equation}\label{eq:cp}
c_{D_\Gamma}:D_\Gamma^s\rightarrow D_\Gamma\qquad\hbox{and}\qquad p_{D_\Gamma}:D_\Gamma^s\rightarrow B_\Gamma
\end{equation}
such that $\frac{\deg(p_{D_\Gamma})}{\deg(c_{D_\Gamma})}=\frac{K_\Gamma}{|\Aut(\Gamma)|\cdot \ell_\Gamma}$. 

The notation used in \cite{comoza} simplified~$p_{D_\Gamma}$ and~$c_{D_\Gamma}$ to~$p_\Gamma$ and~$c_\Gamma$, and we will mostly follow this convention except when it is essential to stress the dependence on which irreducible component of the boundary stratum with a given enhanced level graph is considered. Additionally, we will also use the notation~$p_\Gamma^{[i]}$ (resp.~$p_{D_\Gamma}^{[i]}$) to denote the composition of~$p_\Gamma$ (resp.~$p_{D_\Gamma}$) with the projection of $B_\Gamma = \prod_iB_\Gamma^{[i]}$ onto~$i^{\text{th}}$ factor.

Moreover, \cite[diagram after Prop.~4.4]{comoza} shows that over the interior $D_\Gamma^\circ\subset D_\Gamma$ (we avoid the~$U$ notation from that paper) the open part~$D_\Gamma^{s,\circ}$ admits a finite map~$q_\Gamma$ to a smooth Deligne-Mumford stack~$U_\Gamma^s$ with {\em finite} \'etale Galois covers $p_\Gamma^\Gamma:U_\Gamma^s\rightarrow B_\Gamma^\circ$ and $c_\Gamma^\Gamma:U_\Gamma^s\rightarrow D_\Gamma^\circ$. The degrees of these covers are $K_\Gamma/[R_\Gamma:Tw_\Gamma]$ (i.e.~the number of equivalence classes of prong-matchings) and $[Tw_\Gamma:Tw_\Gamma^s]\cdot |\Aut(\Gamma)|$, respectively. Thus on the open stratum
\begin{equation}\label{eq:cpp}
   c_\Gamma|_{D_\Gamma^{s,\circ}}=c_\Gamma^\Gamma\circ q_\Gamma;\qquad p_\Gamma|_{D_\Gamma^{s,\circ}}=p_\Gamma^\Gamma\circ q_\Gamma,
\end{equation}
so that in particular $\deg(p_\Gamma)/\deg(c_\Gamma)=\deg(p_\Gamma^\Gamma)/\deg(c_\Gamma^\Gamma)$, as given in \cite[eq.~(31)]{comoza}.

For any level~$i$ of~$\Gamma^+$, there is a morphism $\overline f_i:B_\Gamma^{[i]}\rightarrow \prod_{v\in V(\Gamma^{(i)})}\ocH_v$ from the level~$i$ generalized stratum to the product of (as always in our notation, projectivized) strata corresponding to the vertices $v\in V(\Gamma^{(i)})$. Denote 
$$(B_\Gamma^{[i]})^\circ\coloneqq\overline f_i^{-1}\left(\prod_{v\in V(\Gamma^{(i)})}\cH_v\right);\qquad Y\coloneqq \overline f_i((B_\Gamma^{[i]})^\circ)\subset \prod_{v\in V(\Gamma^{(i)})}\cH_v,$$ 
and denote by $f_i:(B_\Gamma^{[i]})^\circ\twoheadrightarrow Y$ the restriction of~$\overline{f}_i$.
\begin{claim}
    \textbf{}Every fiber of~$f_i$ is irreducible and is finitely covered by a proper toric variety.
\end{claim}
\begin{proof}
    In \cite[\S 3]{chcomo}, the authors proved that each component of a fiber of $\pi:\ocH\rightarrow \overline{\cM}_{g,n}$ is finitely covered by a proper toric variety. As discussed in \cite[\S 3.2]{chcomo}, the set of level structures on the underlying dual graph of~$\Gamma$ forms a directed graph, with arrows given by undegeneration. The irreducible components of the fiber over a generic point of~$\pi(D_\Gamma)$ correspond to the terminal elements in this directed graph, i.e.~the level structures on a given dual graph that have minimal possible number of levels. Let~$\Gamma_0$ be one such terminal element such that the generic fiber of $D_\Gamma\rightarrow \pi(D_\Gamma)$ is contained in~$D_{\Gamma_0}$. Then \cite[Prop.~3.3]{chcomo} states that a general fiber of $D_{\Gamma_0}\rightarrow \pi(D_{\Gamma_0})$ is finitely covered by a proper toric variety~$\widetilde{F}_{\Gamma_0}$ for a torus~$T_{\Gamma_0}$ isogenous to~$T^p/T^{np}$. Here $T^p\subset(\bC^*)^{V(\Gamma_0}$ is a torus that acts by rescaling the differentials on each vertex of~$\Gamma_0$ while respecting the global residue condition, whereas $T^{np}\cong (\bC^*)^{L(\Gamma_0)+1}$ is a torus that acts by rescaling the differentials on each level of~$\Gamma_0$. In fact,~$\widetilde{F}_{\Gamma_0}$ is constructed by mimicking the construction of~$D_{\Gamma_0}^s$ in \cite[\S 4]{comoza}, with all objects restricted to a fiber of~$\pi$. Thus~$\widetilde{F}_{\Gamma_0}$ also covers the corresponding fiber~$F_{\Gamma_0}$ of the natural morphism~$B_{\Gamma_0}\rightarrow \pi(D_{\Gamma_0})$ obtained by forgetting the differentials and clutching the corresponding edges of~$\Gamma_0$.

    Since~$\Gamma$ is obtained from~$\Gamma_0$ via a degeneration that preserves the underlying dual graph, a general fiber of~$f_i$ can then be covered by a torus-invariant subvariety of~$\widetilde{F}_{\Gamma_0}$. Denote by~$\Lambda$ the level graph obtained by a maximal degeneration of all but level~$i$ of~$\Gamma$, while keeping the underlying dual graph fixed. Then the torus-invariant subvariety $B_\Lambda\cap F_{\Gamma_0}$ is a fiber of~$f_i$. Recalling that the composition 
    $$Y\rightarrow \prod_{v\in\Gamma^{(i)}}\cH_v\rightarrow \prod_{v\in\Gamma^{(i)}} M_{g(v),n(v)}$$ 
    is injective proves our claim.
\end{proof}

\begin{rem}
    The action of the torus $(\bC^*)^{V(\Gamma^{(i)})}$ by rescaling the differen\-tials on each vertex at level~$i$ of~$\Gamma$ does not respect the global residue condition (GRC). Nonetheless, since GRC is a collection of additive conditions on the residues of the poles, there is a subtorus $T_i^p\subset (\bC^*)^{V(\Gamma^{(i)})}$ that respects the GRC and thus preserves the fiber of~$f_i$. Since~$B_\Gamma^{[i]}$ only remembers the differentials on the vertices of level~$i$ up to an action of~$\bC^*$ that rescales all  differentials simultaneously, the proper toric variety covering a fiber of~$f_i$ is a toric variety for a torus isogenous to $T^p_i/\bC^*$.  This toric variety can be constructed explicitly by following the steps of \cite[\S 3.2]{chcomo}.
\end{rem}

In fact the torus-invariant Weil divisors on each fiber of~$f_i$, being restrictions of divisor classes in~$B_\Gamma$, are~$\bQ$-Cartier. So the covering toric variety can be taken (after normalization if necessary) to be a simplicial toric variety of some fixed combinatorial type. In particular, all fibers of~$f_i$ have the same rational cohomology (see \cite[Thm.~12.3.12]{cls}), so we denote the rational homology fiber by~$F_i$. The following Leray-Hirsch type result will be useful later:
\begin{lem}[Leray-Hirsch]\label{leray_hirsch}
    The cohomology $\cohH^*((B_\Gamma^{[i]})^\circ)$ is isomorphic, as a vector space, to $\cohH^*(Y)\otimes \cohH^*(F_i)$.
\end{lem}
\begin{proof}
    While we don't claim that~$f_i$ is a fiber bundle or a Serre fibration, we can nonetheless mimic the proof of Leray-Hirsch Theorem using the spectral sequence $E_2^{p,q}=\cohH^p(Y,R^q(f_i)_*\bQ)\implies \cohH^{p+q}((B_\Gamma^{[i]})^\circ,\bQ)$. Since~$f_i$ is proper, $R^q(f_i)_*\bQ$ is a local system with fiber~$\cohH^q(F_i)$ for each~$q$, which is non-trivial only for~$q$ even. We claim that $R^q(f_i)_*\bQ$ is the trivial local system. This is clear because the cohomology of a simplicial toric variety is generated by torus-invariant subvarieties, which, in our context, are all given by restriction of subvarieties of~$(B_\Gamma^{[i]})^\circ$. Indeed, the torus-invariant subvarieties in each fiber are in one-to-one correspondence with the degenerations of~$\Gamma^{(i)}$ that keep the underlying dual graph fixed. Now one can follow the usual proof of Leray-Hirsch Theorem (e.g.~\cite[Thms.~5.9 and 5.10]{mccleary_spectralsequences})
to see that the spectral sequence degenerates immediately at $E_2$~page.
\end{proof}

\subsection*{The normal bundle to the boundary strata}
We now recall from \cite[\S~7]{comoza} the formula for the normal bundle to (any irreducible component of) a vertical boundary divisor $D_\Gamma\subset\ocH$, for $\Gamma\in \LG_{1,0}$. Suppose~$e$ is an edge of~$\Gamma$ and~$\kappa_e$ is the enhancement on it. Denoting by~$\psi_{e^\pm}$ the $\psi$-classes associated to the half edges~$e^{\pm}$ that constitute the edge~$e$, the normal bundle $N_{\Gamma}\coloneqq N_{D_\Gamma}$ of~$D_\Gamma$ in~$\ocH$ is given by \cite[eq.~(62)]{comoza}:
\begin{equation}\label{eq:normal_bundle}
    c_1(N_\Gamma) = -\frac{\kappa_e}{\ell_{\Gamma}}(\psi_{e^+}+\psi_{e^-})-\frac{1}{\ell_\Gamma}\sum_{\Delta\in \LG_{2,e}^\Gamma}\sum_{D_\Delta}\ell_{\Delta,a_{\Delta,\Gamma}}D_\Delta\quad\hbox{in\ }CH^1(D_\Gamma).
\end{equation}
Here~$\LG_{2,e}^\Gamma$ is the collection of enhanced level graphs with 3 levels obtained via a degeneration of~$\Gamma$ such that the edge~$e$ becomes long (i.e.~goes from level~$0$ to level~$-2$), $a_{\Delta,\Gamma}\in \{1,2\}$ is the index such that the undegeneration of $a_{\Delta,\Gamma}$'th level passage of~$\Delta$ is not equal to~$\Gamma$, and $\ell_{\Delta,a_{\Delta,\Gamma}}$ is the lcm of the enhancements of the edges crossing $a_{\Delta,\Gamma}$'th level passage. We use the convention that~$e^+$ is the half edge on the higher level and~$e^{-}$ is that on the lower level.

Denote by $\LG_{2,e}^{\Gamma,[i]}\subset \LG_{2,e}^\Gamma$ the subset of level graphs with 3 levels that are obtained by degeneration of the level~$i$ of~$\Gamma$. Then each $\Lambda\in \LG_{2,e}^{\Gamma,[i]}$ defines a divisor in~$B_\Gamma^{[i]}$ (and subsequently in~$B_\Gamma$), which we also denote by~$D_\Lambda$ by abuse of notation, such that $p_\Gamma^*(D_\Lambda)$ and $c_\Gamma^*(D_\Lambda)$ differ by multiplication by a positive constant (see \cite[Prop. 4.7]{comoza}). Thus there exists a divisor class $$\nu_\Gamma\coloneqq -\frac{\kappa_e}{\ell_{\Gamma}}(\psi_{e^+}+\psi_{e^-})-\frac{1}{\ell_\Gamma}\sum_{\Delta\in \LG_{2,e}^\Gamma}\widehat\ell_{\Delta,a_{\Delta,\Gamma}}D_\Delta\in \CH^1(B_\Gamma)$$ such that $p_\Gamma^*\nu_\Gamma=c_\Gamma^*c_1(N_\Gamma)$ (as mentioned above, the coefficients $\ell_{\Delta,a_{\Delta,\Gamma}}$ and $\widehat\ell_{\Delta,a_{\Delta,\Gamma}}$ are related by \cite[Prop. 4.7]{comoza}). Additionally, the classes 
$$\begin{aligned}\nu_\Gamma^{\top}&\coloneqq -\frac{\kappa_e}{\ell_{\Gamma}}\psi_{e^+}-\frac{1}{\ell_\Gamma}\sum_{\Delta\in \LG_{2,e}^{\Gamma,[0]}}\widehat\ell_{\Delta,a_{\Delta,\Gamma}}D_\Delta\quad\hbox{and}\\ \nu_\Gamma^{\bot}&\coloneqq -\frac{\kappa_e}{\ell_{\Gamma}}\psi_{e^-}-\frac{1}{\ell_\Gamma}\sum_{\Delta\in \LG_{2,e}^{\Gamma,[-1]}}\widehat\ell_{\Delta,a_{\Delta,\Gamma}}D_\Delta\end{aligned}$$ 
are pullbacks from~$B_\Gamma^{\top}$ and~$B_\Gamma^\bot$, respectively.

This normal bundle formula generalizes to the case $D_\Lambda\subset D_\Gamma$, where $\Gamma\in \LG_{L,0},\ \Lambda\in \LG_{L+1,0},$ and~$\Lambda$ is a degeneration of~$\Gamma$ that introduces at least one extra edge~$e$. Indeed, if undegeneration of the $i$'th level passage of~$\Lambda$ gives~$\Gamma$, it follows that  the class $N_{\Lambda/\Gamma}=N_{D_\Lambda/D_\Gamma}$ in~$\CH^1(D_\Lambda)$ is given by:
\begin{equation}\label{eq:general_normal_bundle}
    c_1(N_{\Lambda/\Gamma})=-\frac{\kappa_e}{\ell_{\Lambda,-i+1}}(\psi_{e^+}+\psi_{e^-})-\frac{1}{\ell_{\Lambda,-i+1}}\sum_{\Delta\in \LG_{L+2,e}^\Lambda}\sum_{D_\Delta}\ell_{\Delta,a_{\Delta,\Lambda}}D_\Delta\,.
\end{equation}
Here~$\ell_{\Lambda,i}$ is the $\operatorname{lcm}$ of the enhancements on the edges of~$\Lambda$ that cross the $i$'th level passage, $\LG_{L+2,e}^\Lambda$ is the set of all level graphs with~$L+2$ levels such that the edge~$e$ becomes longer, $a_{\Delta,\Lambda}$ is such that the undegeneration of~$\Delta$ that smoothens the $a_{\Delta,\Lambda}$'th level passage is not equal to~$\Gamma$, and $\ell_{\Delta,a_{\Delta,\Lambda}}$ is the lcm of the enhancements of the edges crossing $a_{\Delta,\Gamma}$'th level passage. Analogously, $p_\Gamma^*\nu_{\Lambda/\Gamma}=c_\Gamma^*c_1(N_{\Lambda/\Gamma})$ for some tautological divisor class $\nu_{\Lambda/\Gamma}=\nu_{\Lambda/\Gamma}^{[i]}+\nu_{\Lambda/\Gamma}^{[i-1]}$.

While \cite[Thm. 7.1, Prop. 7.5]{comoza} gives a reformulation of the normal bundle formula that is also applicable to the case where the degeneration of~$\Gamma$ to~$\Lambda$ keeps the underlying dual graph fixed (so introduces no new edge), we will not need that for our purposes.

\subsection*{The tautological ring of the strata}
We finally recall from \cite{comoza} the definition of the (biggest version) of the tautological ring of the strata.
\begin{defn}\label{df:taut}
The tautological rings of strata of differentials are the smallest collection of $\bQ$-algebras $R^*(\ocHgnm)\subset \CH^*(\ocHgnm)$, for all $g,n,\mu$ which:
\begin{itemize}
    \item contain the~$\psi$ classes for each marked point,
    \item is closed under the pushforward via the map forgetting a marked point of order~$0$,
    \item is closed under the map $(c_{D_\Gamma})_*(p_{D_\Gamma})^*$ for each $\Gamma^+\in \LG_{L,H}$ and each irreducible component~$D_\Gamma$ of the boundary with this enhanced level graph.
\end{itemize}
\end{defn}
Dawei Chen \cite{chentautological} showed that~$\RH^*(\cHgnm)$ is generated by the Hodge class~$\lambda$ and the $\psi$-classes of the points corresponding to simple poles. In particular, all tautological classes on a non-meromorphic stratum~$\cHgnm$ are proportional to some power of~$\lambda$. Moreover, in~\cite{chennonvarying} he showed that the class of a certain ample line bundle vanishes on any (strictly) meromorphic stratum, which in particular implies that all meromorphic strata are affine.

The cohomological tautological ring $\RH^*(\ocHgnm)\subset \cohH^*(\ocHgnm)$ is then defined as the image of~$R^*(\ocHgnm)$ under the cycle class morphism $\CH^*(\ocHgnm)\rightarrow \cohH^*(\ocHgnm)$.

It was proven in \cite{Devkotacohomology} that the Chow (as well as cohomology) ring of the strata of multi-scale differentials in genus~$0$ is isomorphic to the tautological ring and is generated by the boundary divisors.
 
Let us also recall Theorem 1.5 of \cite{comoza}. Note that the theorem is stated there for the smaller versions of tautological rings where in the third bullet point in \Cref{df:taut}, only the union of all irreducible components~$D_\Gamma$ of the stratum with a given enhanced level graph $\Gamma^+$ is taken. However, the proof also works for the larger version of tautological ring, yielding the following.
\begin{thm}[{\cite[Thm.~1.5]{comoza}}]\label{thm: fin_gen_taut_ring}
    For each $\mu,g,n$, a finite set of additive generators of~$R^*(\ocHgnm)$ is given by the classes
    \[(\iota_{D_\Gamma})_*(c_{D_\Gamma})_*\left(\prod_{j=0}^{-L}(p_{D_\Gamma}^{[j]})^*\alpha_j\right),\]
    where~$D_\Gamma$ run over the set of all irreducible components of all boundary strata,~$\alpha_j$ is a monomial in $\psi$-classes supported on level~$j$ of the graph~$\Gamma$, and $\iota_{D_\Gamma}:D_\Gamma\hookrightarrow\ocHgnm$ is the inclusion.
\end{thm}
As a corollary, we have the following result, which is essentially contained in \cite{comoza}:
\begin{cor}\label{cor: tautological_kunneth}
    For any $\alpha\in \RH^*(\ocHgnm)$, any $\Gamma\in \LG_{L,0}$, and any irreducible boundary stratum~$D_\Gamma$ in~$\ocH$, the K\"unneth decomposition of $(p_{D_\Gamma})_*(c_{D_\Gamma})^*(\iota_{D_\Gamma})^*\alpha$  on~$B_\Gamma$ consists of only tautological classes.
\end{cor}
The analog of this statement for~$\overline\cM_{g,n}$ was proven in \cite[Prop.~12]{graberpandharipande}.
\begin{proof}
    Given a level graph $\Lambda\in \LG_{L,H}$ and a class $\alpha\in \CH^*(D_\Lambda)$, the pushforward to~$\ocHgnm$ and then pullback to~$D_\Gamma$ is computed in \cite[Prop.~8.1]{comoza}. As all tautological classes are generated in this way by \Cref{thm: fin_gen_taut_ring}, the corollary is proven.
\end{proof}

\section{Non-trivial classes in $\cohH^3(\ocH)$}\label{section:nontrivial_H3}
In this section we use a method from \cite[\S~6]{calapa}, building upon some of the ideas from \cite{graberpandharipande}, to construct non-trivial degree 3 cohomology classes on compactified strata. The basic building block will be the non-trivial first (co)homology of the residueless genus~$1$ strata. By \Cref{prop:genusLT} these are complex curves that have a connected component of positive genus, except in a few cases. For all constructions below, we note that in any genus and for any~$\mu$, we have $m_1\ge 0$. Indeed, as we are ordering~$m_1$ in a non-increasing way, otherwise we would have $-1\ge m_1\ge \dots \ge m_n$, so that in particular we must have $g=0$ but then $\sum m_i=-2$ implies $n=2$, which is unstable.

In any stratum consider an irreducible component~$D_\Gamma$ of the boundary divisor with level graph
\begin{figure}[H]
  \centering
\begin{tikzpicture}[strata,baseline]
  \node[vertex] (1) at (0,0) {$1$};
  \node[vertex] (g1) at (-1.5,2) {$g_1$};
  \node[vertex] (gk) at ( 1.5,2) {$g_k$};
  \node at ($(g1)!0.5!(gk)$) {$\ldots$};
  \draw[thick] (1) -- (g1);
  \draw[thick] (1) -- (gk);
  \draw (1.south) -- ++(0,-0.6)
        node[leglabel,anchor=north] {$z_1$};
\end{tikzpicture}
\caption{Level graph with a genus~$1$ residueless stratum in the bottom level}
\label{fig: divisor_residueless_bottom_level}
\end{figure}
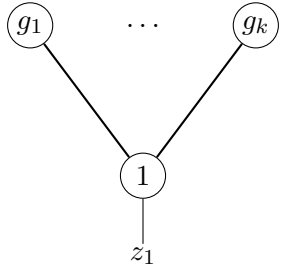
\noindent where we have not labeled the location of the marked zeroes or poles $z_2,\dots,z_n$, but all of them will be on the top level components. Moreover, we will require all marked poles to be contained in the rightmost top level component, of genus~$g_k$. We will in due course investigate whether it is possible to assign~$g_j$, distribute~$z_j$, and assign enhancements in such a way that such a boundary stratum in~$\ocHgnm$ exists. 

Since top level vertices of genera $g_1,\dots,g_{k-1}$ parameterize holomorphic strata, they each impose a global residue condition at the corresponding unique vertical edge: that the residue is zero at the corresponding pole at the bottom. Then by the residue theorem on the bottom level, the residue at the edge connecting to the genus~$g_k$ vertex is also zero. Thus the bottom level parameterizes the residueless stratum $\cR_{1,1+k}(m_1,-b_1,\dots,-b_k)$, where $b_j=\kappa_j+1$ in terms of the enhancement on the corresponding vertical edge. Thus by \Cref{prop:genusLT} unless $(m_1,-b_1,\dots,-b_k)$ is one of the exceptional cases given in \eqref{eq:exceptionsLT}, this genus~$1$ residueless stratum has a connected component~$\cR$ that is a curve of positive genus.  We thus consider some connected component~$\ocH$ of~$\ocHgnm$ and some irreducible boundary component~$D_\Gamma$ such that the bottom level is that connected component~$\cR$ (this is possible, because we just take a boundary point in~$\cR$, and it can be smoothed into the open stratum~$\cHgnm$).
\begin{rem}
For the cases $\mu=(2g-2)$ or $(g-1,g-1)$ of holomorphic strata that have a hyperelliptic connected component, the component~$D_\Gamma$ such that the bottom level~$\cR$ in the graph above has positive genus is contained in one of the non-hyperelliptic components of~$\ocHgnm$. Indeed,~$\cM_{0,2g+2}$ surjects onto the hyperelliptic component $\cH^{hyp}\subset \cH_{g,1}(2g-2)$, with $\cM_{0,2g+2}/S_{2g+1}\cong \cH^{hyp}$. Using admissible covers, this surjection extends to a morphism $\overline{\cM}_{0,2g+2}/S_{2g+1}\rightarrow \cH^{hyp}_{DM}\subset \overline\cM_{g,1}$, where~$\cH^{hyp}_{DM}$ denotes the closure of~$\cH^{hyp}$ in~$\overline\cM_{g,1}$. The image of~$D_\Gamma$ in~$\cH_{DM}$ can be realized as the image of a boundary stratum in~$\overline{\cM}_{0,2g+2}$, and so cannot have non-trivial first cohomology. On the other hand, if $\mu=(g-1,g-1)$ and $(C,p,q)\in \cH^{hyp}$, then~$p,q$ are conjugate under the hyperelliptic involution. But this is not possible for any~$2:1$ admissible cover from a stable curve with dual graph~$\Gamma$ to a curve of genus~$0$.
\end{rem}

We now prove that this constructs non-trivial cohomology classes on compactified strata.
\begin{lem}\label{lemma: nontrivial_h3}
If we can arrange a choice of $g_1,\dots,g_k,\kappa_1,\dots,\kappa_k$, and a choice of a partition of $z_2,\dots, z_n$ among the top level components with all poles contained in the genus~$g_k$ component, such that~$\cR$ has positive genus and such that furthermore $\Aut(\Gamma)=1$, then $\cohH^1(D_\Gamma)\ne 0$ and also $\cohH^3(\ocH)\ne 0$. 
\end{lem}
In fact, the proof below will show that $\cohH^1(B_\Gamma^\bot)$ injects into~$\cohH^3(\ocH)$, so that the rank of~$\cohH^3(\ocH)$ is equal to at least~$2g(\cR)$. We do not pursue this improvement here, as we will eventually get exponential growth order anyway.
\begin{proof}
Indeed, if $\Aut(\Gamma)=1$, then by \cite[\S~4]{comoza} the map $c_\Gamma^\Gamma$ in \eqref{eq:cpp} is an isomorphism, and thus $p_\Gamma^\Gamma\circ (c_\Gamma^\Gamma)^{-1}\colon D_\Gamma^\circ\to B_\Gamma^\circ$ is a  finite morphism. We view this as a generically finite rational map $D_\Gamma\dashrightarrow B_\Gamma$, and take a resolution $f_\Gamma:X_\Gamma\rightarrow B_\Gamma$ such that~$X_\Gamma$ has only finite quotient singularities and $g_\Gamma:X_\Gamma\rightarrow D_\Gamma$ is a birational morphism, biregular over~$D_\Gamma^\circ$. Then by \cite[Thm.~7.8]{kollarshafarevichmaps}, $\cohH^1(X_\Gamma)\cong \cohH^1(D_\Gamma)$, and thus~$\cohH^1(B_\Gamma)$ injects into~$\cohH^1(D_\Gamma)$; since $\cohH^1(B_\Gamma^{[-1]})\neq 0$, it follows that~$\cohH^1(D_\Gamma)$ is also non-trivial, and we take a non-zero class~$\alpha$ in it.

To show that this also gives a non-zero class in~$\cohH^3(\ocH)$, we use an idea from \cite[\S~6]{calapa}, by defining such a class by pushforward under the inclusion $\iota:D_\Gamma\hookrightarrow\ocH$, and then showing that it is non-zero by restricting back to~$D_\Gamma$. By abuse of notation, from now on we denote the class $(g_\Gamma)_*(f_\Gamma)^*(1\otimes\alpha)$ also by $1\otimes\alpha$. 

We then have $\iota^*\iota_*(1\otimes\alpha)=c_1(N_{D_\Gamma})\cdot(1\otimes \alpha)$. By \eqref{eq:normal_bundle} and the discussion that follows, $c_\Gamma^*(c_1(N_{D_\Gamma}))=p_\Gamma^*(\nu_\Gamma^\top+\nu_\Gamma^\bot)$, where~$\nu_\Gamma^\top$ (resp.~$\nu_\Gamma^\bot$) is a tautological class on~$B_\Gamma$ pulled back from~$B_\Gamma^\top$ (resp.~$B_\Gamma^\bot$). Using K\"unneth decomposition, it is enough to show that 
$$\nu_\Gamma^\top= -\frac{\kappa_e}{\ell_{\Gamma}}\psi_{e^+}-\frac{1}{\ell_\Gamma}\sum_{\Delta\in \LG_{2,e}^{\Gamma,[0]}}\sum_{D_\Delta}\widehat\ell_{\Delta,a_{\Delta,\Gamma}}D_\Delta$$
is non-trivial on~$B_\Gamma^\top$ (where~$e$ is an edge in~$\Gamma$). This is indeed clear since the set $\LG_{2,e}^{\Gamma,[0]}$ is non-empty: the vertex adjacent to~$e^+$ can be moved up without changing the underlying dual graph, while~$\psi_{e^+}$ is a nef class, so it cannot be anti-effective.
\end{proof}
To prove that~$\cohH^3$ is non-zero it thus suffices to show that the enhanced level graph~$\Gamma$ can be realized in the boundary, which we will show is the case for $m_1\ge 9$. To get an exponential lower bound for the dimension of the cohomology, we will show that the classes thus constructed, starting from different boundary strata, are linearly independent.
\begin{proof}[Proof of \Cref{thm:H3}]
Since $g\ge 3$, or $g\ge 2$ and $n\ge 3$, the level graph~$\Gamma$ as above, with $k=2$ and $g_1=1$, that is with two vertices of genera~$1$ and~$g-2$ in the top level, is the dual graph of a stable curve. In particular, the marked points $z_2,\ldots,z_n$ are all on the genus $g-2$ component in the top level (where $g-2\ge 0$ by assumption, and $g-2=0$ is possible). Then the bottom level parameterizes the  residueless genus~$1$ stratum $\cR_{1,3}(m_1,-2,2-m_1)$, and by~\Cref{prop:genusLT}, precisely for $m_1\ge 9$ this stratum has a connected component that is a curve of positive genus, and thus $\cohH^3(\ocH)\neq 0$. 
        
We now proceed to show that~$\dim \cohH^3$ grows at least exponentially quickly as~$m_1$ increases, for $m_2,\dots,m_n$ fixed.

Denote by $\LG_{1,0}^\cR$ the collection of enhanced level graphs with two levels and no horizontal edge, such that the bottom level is a residueless genus~$1$ stratum  (i.e.~these are the graphs of the form in \Cref{fig: divisor_residueless_bottom_level}). Then for any graph $\Gamma\in \LG_{1,0}^\cR$, we have seen that $\cohH^1(B_\Gamma^\perp)$ injects into~$\cohH^3(\ocHgnm)$. The claim below will show that the images of these maps for all $\Gamma\in \LG_{1,0}^\cR$ are linearly independent, so it remains to obtain an exponential lower bound for the number of such~$\Gamma$ that are realizable within a given stratum~$\ocHgnm$ (i.e.~such that the conditions of \Cref{lemma: nontrivial_h3} are satisfied).

We write $m_1=M+2N-2$, and consider any partition of~$N$ into ordered positive distinct integer parts, that is $N=t_1+\dots+t_k$, with $t_1>\dots> t_k>0$ and $t_i\in \bZ$.  By using Hardy-Ramanujan's growth order $p(N)\sim (4N\sqrt{3})^{-1}\exp (\pi\sqrt{2N/3})$ for the number of all partitions of~$N$ \cite[eq.~(5.22)]{hardyramanujan}, and using the bijection between partitions into odd parts and into distinct parts, the growth order of the number of partitions of~$N$ into distinct parts is $\exp(\pi\sqrt{N/3})$ \cite[eq.~(5.27)]{hardyramanujan}.

For each such partition, we consider the two-level graph $\Gamma\in \LG_{1,0}^{\cR}$ of genus~$g$ with only vertical edges and~$k+1$ vertices of genera $t_1,\dots,t_k$ and $(M+m_2+\dots +m_n-2)/2=g-N-1$ in the top level, and a vertex representing residueless stratum of genus~$1$ at the bottom level. The~$k$ vertices of genera $t_1,\ldots,t_k$ have no marked legs and are each connected by an edge with enhancement~$2t_i-1$ to the bottom level vertex. The only vertices to have marked zeroes or poles will be the bottom vertex, with the unique zero~$z_1$ of order $m_1=M+2N-2$, and the top level vertex of genus~$g-N-1$, carrying zeroes and poles $z_2,\dots,z_n$. Since all~$t_i$ are distinct, they cannot be permuted by automorphisms. Thus unless~$n=1$, so that the top level vertex of genus~$g-N-1$ has no marked point,~$\Aut(\Gamma)=1$. For the case $\mu=(2g-2)$, we take~$N=g-1$ instead, and remove the top level vertex of genus~$g-N-1$ from~$\Gamma$. Thus the graph~$\Gamma$ obtained from any such partition of~$N$ satisfies the conditions of \Cref{lemma: nontrivial_h3}. The claim below shows that each~$\Gamma$ contributes at least one (independent) dimension to~$\cohH^3(\ocHgnm)$, completing the proof.
\end{proof}

\begin{claim}\label{claim:indepH3}
    The morphism $\bigoplus_{\Gamma\in \LG_{1,0}^\cR}\bigoplus_{D_\Gamma} \cohH^1(B_\Gamma^\bot)\rightarrow \cohH^3(\ocHgnm)$ is injective.
\end{claim}
\begin{proof}
Suppose $\sum_\Gamma\sum_{D_\Gamma} a_{D_\Gamma}\alpha_{D_\Gamma}=0$, where $\alpha_{D_\Gamma}$ is a class in~$\cohH^3(\ocH)$ given by the pushforward of a class in~$\cohH^1(B_\Gamma^\bot)$ as above, and the sum runs over all $\Gamma\in \LG_{1,0}^\cR$ and for each such~$\Gamma$ over all irreducible components~$D_\Gamma$ of the stratum with this enhanced level graph. Then for a given $\Lambda\in \LG_{1,0}^\cR$ and a given~$D_{\Lambda}$, denoting by~$\iota_{D_\Lambda}$ the inclusion $D_{\Lambda}\hookrightarrow \ocH$, we obtain 
$$
  \sum_{\Gamma\in \LG_{1,0}^\cR}\sum_{D_\Gamma}a_{D_\Gamma}(\iota_{D_\Lambda})^*\alpha_{D_\Gamma}=0\in \cohH^3(D_\Lambda)\,.
$$ 
Note now that $D_\Gamma\cap D_\Lambda=\emptyset$ for any $\Gamma\ne \Lambda\in \LG_{1,0}^\cR$, since the bottom level of each such graph is a curve, which is a residueless stratum, and thus only has degenerations of noncompact type. Moreover, any other irreducible boundary divisor with the same enhanced level graph~$\Lambda$ is disjoint from~$D_\Lambda$. Thus we obtain the equality $c_1(N_{D_\Lambda})\cdot a_{D_\Lambda}\alpha_{D_\Lambda}=0\in \cohH^3(D_\Lambda)$. As seen in the proof of \Cref{lemma: nontrivial_h3}, this can happen only if $a_{D_\Lambda}\alpha_{D_\Lambda}=0$, and thus the claim follows from the injectivity of the map $\cohH^1(B_\Lambda^\bot)\rightarrow \cohH^3(\ocH)$ for the~$B_\Lambda$ corresponding to a given fixed~$D_\Lambda$.
\end{proof}
We now state some corollaries of \Cref{thm:H3}.
\begin{cor}\label{cor:nontrivialH5}
    Suppose $g\ge 2$ and $m_1\ge 9$ (and if $g=2$, then suppose also $n\ge 3$), or $g\ge 3$ and $n\ge 3$ and $m_1+m_2+m_3\ge 9$ and $m_3\ge 0$, or $g\ge 3$ and $n\ge 4$ and $m_1+m_2+m_3\ge 9$, or $g=2$ and $n\ge 5$ and $m_1+m_2+m_3\ge 9$. Then $\cohH^5(\ocHgnm)\ne 0$.
\end{cor}
\begin{proof}
If $m_1\ge 9$, then~$\mu$ satisfies the hypothesis of \Cref{thm:H3}, so the non-triviality of~$\cohH^5(\ocHgnm)$ follows from~$\cohH^3(\ocHgnm)\ne 0$ by applying Hard Lefschetz Theorem. 

If $n\ge 4, m_1\le 8,$ and $m_1+m_2+m_3\ge 9$ but~$m_3<0$, it follows that~$m_2$ is positive. Both in this case and in the case
$n\ge 3$,~$m_3\ge 0$ and $m_1+m_2+m_3\ge 9$, the tuple $\mu_0\coloneqq(m_1+m_2+m_3,m_4,\ldots,m_n)$ satisfies the hypothesis of \Cref{thm:H3}.  Then for the level graph~$\Gamma$ given below (where the marked points $z_4,\dots,z_n$ are at the top level), its top level vertex parameterizes the stratum $\ocH_{g,n-2}(\mu_0)$, and the bottom level parameterizes a $1$-dimensional stratum of differentials in genus~$0$ (this is where we need~$n\ge 4$ when~$m_3<0$ -- otherwise the top level would parameterize holomorphic differentials, so the global residue condition would impose a non-trivial condition on the differentials in the bottom level).
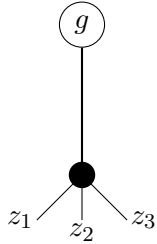
\begin{figure}[H]
  \centering
\begin{tikzpicture}[strata,baseline]
    \fill (0,0) circle (5pt);
  \coordinate (bottom) at (0,0);
  \node[vertex] (g)    at (0,2) {$g$};
  \draw[thick] (bottom) -- (g);
  \draw (bottom.south west) -- ++(-0.6,-0.6)
        node[leglabel,anchor=east] {$z_1$};
  \draw (bottom.south) -- ++(0,-0.6)
        node[leglabel,anchor=north] {$z_2$};
  \draw (bottom.south east) -- ++(0.6,-0.6)
        node[leglabel,anchor=west] {$z_3$};
\end{tikzpicture}
\end{figure}

Following as in the proof of \Cref{lemma: nontrivial_h3}, we see that it is enough to show that~$\nu_\Gamma^\bot$ is non-trivial (where, as discussed in \eqref{eq:normal_bundle},~$\nu_\Gamma^\bot$ is the bottom level contribution to the normal bundle of~$D_\Gamma$). But this is clear since $B_\Gamma^\bot\cong \overline{\cM}_{0,4}$, so~$\psi_{e^-}$ has strictly positive degree on~$B_\Gamma^\bot$, and thus~$\nu_\Gamma^\bot$ has strictly negative degree on~$B_\Gamma^\bot$.
\end{proof}
\begin{rem}
    In the proof of this corollary, it is essential to have at least three points colliding, as otherwise the push-pull can be easily checked to be trivial. Indeed suppose $g\geq 5$ and $n=2$, and consider a codimension 2 boundary stratum corresponding to the enhanced level graph~$\Lambda$ below:
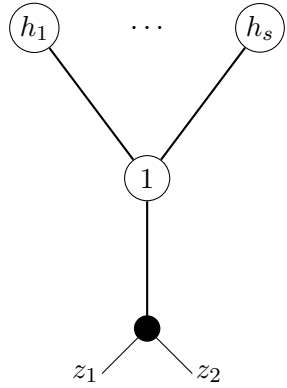
\begin{figure}[H]
  \centering
\begin{tikzpicture}[strata,baseline]
  \fill (0,0) circle (5pt);
  \coordinate (bottom) at (0,0);
  \node[vertex] (1) at (0,2) {$1$};
  \node[vertex] (h1) at (-1.5,4) {$h_1$};
  \node[vertex] (hs) at (1.5,4) {$h_s$};
  \node at ($(h1)!0.5!(hs)$) {$\ldots$};
  \draw[thick] (bottom) -- (1);
  \draw[thick] (1) -- (h1);
  \draw[thick] (1) -- (hs);
  \draw (bottom.south west) -- ++(-0.6,-0.6)
        node[leglabel,anchor=east] {$z_1$};
  \draw (bottom.south east) -- ++(0.6,-0.6)
        node[leglabel,anchor=west] {$z_2$};
\end{tikzpicture}
\end{figure}

Then $D_\Lambda\subset D_{\Gamma_1}\cap D_{\Gamma_2}$ for two irreducible boundary divisors~$D_{\Gamma_1}$ and~$D_{\Gamma_2}$ such that~$\Lambda$ is the degeneration of the top level of~$\Gamma_1$ and of the bottom level of~$\Gamma_2$. Note that~$B_\Lambda^{[-1]}$ is as usual a residueless stratum that is abstractly a curve of positive genus, while~$B_\Lambda^{[-2]}$ is 0-dimensional. So for the edge~$e$ connecting the two lower levels, the normal bundle formula \eqref{eq:general_normal_bundle} gives $c_1(N_{\Lambda/\Gamma_2})=-\frac{\kappa_e}{\ell_{\Lambda,2}}\psi_{e^+}=-\psi_{e^+}$ since~$\psi_{e^-}=0$ and~$\Lambda$ has no degeneration that makes~$e$ longer. Since $$c_2(N_\Lambda)=c_1(N_{\Lambda/\Gamma_1})\cdot c_1(N_{\Lambda/\Gamma_2})\,,$$ for any $\alpha\in \cohH^1(B_\Lambda^{[-1]})\hookrightarrow \cohH^1(D_\Lambda)$, we must have $c_2(N_\Lambda)\cdot \alpha=0$ since $\psi_{e^+}\cdot \alpha=0$ (recall that~$B_\Lambda^{[-1]}$ has complex dimension~$1$, while $\psi_{e^+}\cdot \alpha\in \cohH^3(B_\Lambda^{[-1]})$).
\end{rem}

The next corollary gives an application to the Chow ring of~$\ocH$:
\begin{cor}\label{cor:uncountable_chow}
    If~$\mu$ satisfies the hypothesis of \Cref{cor:nontrivialH5}, then the Chow ring~$\CH^*(\ocHgnm)$ is uncountable, and in particular~$R^*(\ocHgnm)$ is a proper subset of~$\CH^*(\ocHgnm)$.
\end{cor}
\begin{proof}
    Since the tautological ring~$R^*(\ocHgnm)$ is a finite-dimensional $\bQ$-vector space, this follows from the argument parallel to \cite[Cor.~1.6]{calapa}, since~$\CH^*(\ocHgnm)$ is uncountable by \cite[Thm.~7.1]{calapa}.
\end{proof}
While this corollary guarantees existence of uncountably many non-tautological classes in the Chow ring of~$\ocH$, it a priori does not say anything about non-tautological classes in the even degree cohomology of~$\ocH$, which we will discuss in the next section.

\section{Non-tautological even degree cohomology classes.}\label{section:nontautological_even}
Inspired by the construction in the previous section, 
we construct non-tautological cohomology classes in $\cohH^{2,2}(\ocH,\bC)\cap \cohH^4(\ocH,\bQ)$.

Consider an irreducible component~$D_\Gamma$ of the boundary divisor defined by the following level graph~$\Gamma$:
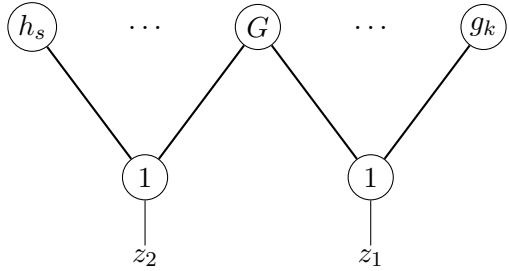
\begin{figure}[H]
  \centering
\begin{tikzpicture}[strata,baseline]
  \node[vertex] (1l) at (0,0) {$1$};
  \node[vertex] (G) at (-1.5,2) {$G$};
  \node[vertex] (gk) at (1.5,2) {$g_k$};
  \node at ($(G)!0.5!(gk)$) {$\ldots$};
  \node[vertex] (1r) at (-3,0) {$1$};
  \node[vertex] (hs) at (-4.5,2) {$h_s$};
  \node at ($(hs)!0.5!(G)$) {$\ldots$};
  \draw[thick] (1l) -- (G);
  \draw[thick] (1l) -- (gk);
  \draw[thick] (G) -- (1r);
  \draw[thick] (1r) -- (hs);
  \draw (1l.south) -- ++(0,-0.6)
        node[leglabel,anchor=north] {$z_1$};
  \draw (1r.south) -- ++(0,-0.6)
        node[leglabel,anchor=north] {$z_2$};
\end{tikzpicture}
\end{figure}
Here, as before, all the marked zeroes and poles except for~$z_1$ and~$z_2$ are assumed to be attached to vertices of top level, and moreover we assume all marked poles to be attached to the ``middle'' top level vertex with genus labeled $G=g_1=h_1$. The global residue conditions imposed at the bottom level by the top level vertices of genera $g_2,\dots,g_k$ are that there are no residues at the preimages of the nodes on the right bottom elliptic curve connecting to them, and thus by the residue theorem the right bottom vertex parameterizes a residueless genus~$1$ stratum $\cR_{1,1+k}(m_1,-a_1,\ldots,-a_k)$. Similarly, the bottom left vertex parameterizes a residueless stratum $\cR_{1,1+s}(m_2,-b_1,\ldots,-b_s)$. We will want both of these strata to have a connected component that is a curve of positive genus as provided by \Cref{prop:genusLT}. Then for non-zero classes
$$\alpha_1\in \cohH^1(\cR_{1,1+k}(m_1,-a_1,\ldots,-a_k)),\ \alpha_2\in \cohH^1(\cR_{1,1+s}(m_2,-b_1,\ldots,-b_s))\,,$$ 
pulling back $\alpha_1\otimes\alpha_2$ to~$B_\Gamma^\bot$ defines a non-tautological (by \Cref{thm: fin_gen_taut_ring}) class $\alpha\in \cohH^2(B_\Gamma)$, and moreover choosing $\alpha_1\in \cohH^{1,0}$ and $\alpha_2\in \cohH^{0,1}$ or vice versa yields $\alpha\in \cohH^{1,1}(B_\Gamma)$.
\begin{lem}\label{lem:nontaut_H2_boundary}
Suppose there exists a choice of
\begin{itemize}
    \item $G=g_1=h_1,g_2,\dots,g_k,h_2,\dots,h_s;$
    \item enhancements~$\kappa_e$ on all the edges;
    \item a partition of marked zeroes and poles $z_3,\dots, z_n$ among all top level vertices, such that all poles are on the middle vertex of genus~$G$,
\end{itemize} 
such that 
\begin{itemize}
    \item the above~$\Gamma$ is realized by a boundary stratum in~$\ocHgnm$;
    \item both bottom genus~$1$ residueless strata have a connected component of positive genus; 
    \item and furthermore~$\Aut(\Gamma)=1$.
\end{itemize}
Then there exists an irreducible boundary stratum~$D_\Gamma$ with this enhanced level graph, with $\cohH^2(D_\Gamma)\ne 0$. Furthermore there exists a nonhyperelliptic connected component~$\ocH\subset\ocHgnm$ such that $\RH^4(\ocH)\subsetneq \cohH^4(\ocH)$.
\end{lem}
\begin{proof}
We argue similarly to the proof of \Cref{lemma: nontrivial_h3}, as~$c_\Gamma^\Gamma$ is again an isomorphism. Thus we view $p_\Gamma^\Gamma\circ (c_\Gamma^\Gamma)^{-1}$ as a rational map $D_\Gamma\dashrightarrow B_\Gamma$ that restricts to a finite morphism  $p_\Gamma^\Gamma:D_\Gamma^\circ\rightarrow B_\Gamma^\circ$ on the open stratum. However, birational varieties do not necessarily have isomorphic second cohomology, so we need a more refined argument. Since~$p_\Gamma^\Gamma$ is a finite surjective morphism of smooth orbifolds, the pullback~$(p_\Gamma^\Gamma)^*$ induces an injection 
$$(p_\Gamma^\Gamma)^*:\Gr^W_2\,\cohH^2(B_\Gamma^\circ)\hookrightarrow \Gr^W_2\,\cohH^2(D_\Gamma^\circ)
$$ on the weight 2 component of the mixed Hodge structures (see \Cref{section: open_strata} for a review of how this works). Consequently, on the interior~$D_\Gamma^{s,\circ}$ we have $c_\Gamma|_{D_\Gamma^{s,\circ}}=c_\Gamma^\Gamma\circ q_\Gamma=q_\Gamma$ and $p_\Gamma = p_\Gamma^\Gamma\circ q_\Gamma$. It follows that for the class $\alpha\in \cohH^2(B_\Gamma)$ defined above we have the equality
\[
(c_\Gamma)_*(p_\Gamma)^*(\alpha|_{B_\Gamma^\circ})=(q_\Gamma)_*(q_\Gamma)^*(p_\Gamma^\Gamma)^*(\alpha|_{B_\Gamma^\circ})=\deg(q_\Gamma)(p_\Gamma^\Gamma)^*(\alpha|_{B_\Gamma^\circ})\neq 0
\]
in~$\cohH^2(D_\Gamma^\circ)$. Since $\Gr_2^W\,\cohH^2(D_\Gamma^\circ)$ is given by restriction of~$\cohH^2(D_\Gamma)$ onto~$D_\Gamma^\circ$, it follows that $(c_\Gamma)_*(p_\Gamma)^*\alpha$ is also non-trivial in~$\cohH^2(D_\Gamma)$.

Following as in the proof of \Cref{lemma: nontrivial_h3}  implies that $(\iota_\Gamma)_*(c_\Gamma)_*(p_\Gamma)^*\alpha$ is non-trivial in~$\cohH^4(\ocH)$. Indeed, $$(\iota_\Gamma)^*(\iota_\Gamma)_*(c_\Gamma)_*(p_\Gamma)^*\alpha=c_1(N_\Gamma)\cdot (c_\Gamma)_*(p_\Gamma)^*\alpha.$$ Since $(c_{\Gamma})^*c_1(N_\Gamma)=(p_\Gamma)^*\nu_\Gamma$ for some tautological class~$\nu_\Gamma$ on~$B_\Gamma$, we have by the projection formula
\[
    (p_\Gamma)_*(c_\Gamma)^*\left(c_1(N_\Gamma)\cdot (c_\Gamma)_*(p_\Gamma)^*\alpha\right)
    =\nu_\Gamma\cdot (p_\Gamma)_*(c_\Gamma)^*(c_\Gamma)_*(p_\Gamma)^*\alpha,
\]
Recalling that $p_\Gamma=p_\Gamma^\Gamma\circ q_\Gamma$ on the open stratum, by~\eqref{eq:cpp}, it follows that 
\[
  (p_\Gamma)_*(c_\Gamma)^*(c_\Gamma)_*(p_\Gamma)^*(\alpha|_{B_\Gamma^\circ})=  \deg(p_\Gamma^\Gamma)\cdot (\deg(q_\Gamma))^2\alpha|_{B_\Gamma^\circ}\,,
\] 
and thus on the compactification the class $(p_\Gamma)_*(c_\Gamma)^*(c_\Gamma)_*(p_\Gamma)^*\alpha$ is a scalar multiple of~$\alpha$ plus some linear combination of boundary divisors of~$\partial B_\Gamma$. Consequently, the K\"unneth decomposition of $$(p_\Gamma)_*(c_\Gamma)^*\left(c_1(N_\Gamma)\cdot (c_\Gamma)_*(p_\Gamma)^*\alpha\right)$$ contains a non-tautological class that is a multiple of~$\alpha$, so we obtain a non-tautological algebraic cycle class in~$\cohH^{2,2}(\ocH)$ by \Cref{cor: tautological_kunneth}. 

When $\mu=(g-1,g-1)$, the boundary divisor~$D_\Gamma$ is contained in one of the non-hyperelliptic components of the stratum~$\ocH_{g,2}(g-1,g-1)$, since the marked zeros cannot be conjugate under the hyperelliptic involution.
\end{proof}

To complete the proof of \Cref{thm:H4}, we now again need to show that for $m_1,m_2\ge 9$ such boundary strata are realizable in a given~$\ocHgnm$.

\begin{proof}[Proof of \Cref{thm:H4}]
The existence proof is parallel to the construction in the proof of \Cref{thm:H3}. Indeed, we put all marked zeroes and poles $z_3,\dots,z_n$ at the vertex of genus $G=g_1=h_1=g-4$ at the top level, and make $g_2=h_2=1$, with the corresponding two top level vertices containing no marked zeroes or poles. By the conditions on~$g$ and~$n$ the underlying dual graph is a dual graph of a stable curve. Then the bottom level vertices correspond to the residueless strata $\cR_{1,3}(m_1,-2,2-m_1)$ and $\cR_{1,3}(m_2,-2,2-m_2)$. The graph still has no automorphisms, as the two unmarked genus~$1$ top level vertices are connected to distinct marked bottom level vertices. For $m_1\ge m_2\ge 9$ both residueless strata have a component of positive genus, and we have thus satisfied the conditions of \Cref{lem:nontaut_H2_boundary} to show that $\RH^4(\ocH)\subsetneq \cohH^4(\ocH)$. 

For the growth order, we again simply replace one top level genus~$1$ vertex with an arbitrary partition of~$N$ into distinct genera, which by the last two paragraphs of the proof of \Cref{thm:H3} gives the same exponential lower bound for the growth order of such partitions. That the classes in~$\cohH^4$ obtained in this way via pushforward from distinct graphs~$\Gamma$ are independent follows from the same argument as in the proof of \Cref{claim:indepH3}. Indeed, as in the proof of \Cref{claim:indepH3}, the components of the boundary divisors given by the divisorial level graphs of the type we are now considering are pairwise disjoint (as the residueless strata appearing in the bottom level only admit degenerations of noncompact type), so the argument goes through verbatim.
\end{proof}

The construction of non-tautological classes in~$\cohH^4(\ocHgnm)$ given in the proof above requires $m_2\ge 9$. If~$n$ is large enough, we can of course construct non-tautological classes in~$\cohH^6$ from non-tautological classes in~$\cohH^4$ by colliding marked zeroes and using the Hard Lefschetz Theorem or the argument such as in the proof of \Cref{cor:nontrivialH5}. 

\begin{rem}
This construction can be still further generalized, to the case when $m_1\ge m_2\ge\dots \ge m_k\ge 9$ and $g\ge 2k+1$ (or $g=2k$ and $n\ge k+1$). Indeed, in this case we consider a graph~$\Gamma$ that is still vertical two-level, with a ``central'' top level vertex of genus~$G$ and valence~$k$ containing all marked poles and some zeroes, connected by one edge to each of the~$k$ genus~$1$ bottom level vertices. Each bottom level vertex has the corresponding~$z_j$ among $z_1,\dots,z_k$ attached, and another edge going up to a valence~$1$ top level vertex, which possibly contains  further marked zeroes. This will show that if $m_1\ge m_2\ge\dots \ge m_k\ge 9$, and $g\ge 2k+1$ (or $g=2k$ and $n\ge k+1$), then $\cohH^k(D_\Gamma)\ne 0$ and that thus~$\cohH^{k+2}(\ocH)$ is not purely tautological for some connected component~$\ocH$ of~$\ocHgnm$.
\end{rem}

Our next result will construct non-tautological classes in~$\cohH^6$ in further cases when neither of these arguments applies. Moving beyond the above constructions using boundary divisors, we consider a suitable irreducible component~$D_\Lambda$ of the codimension 2 boundary stratum with the following enhanced level graph~$\Lambda$:
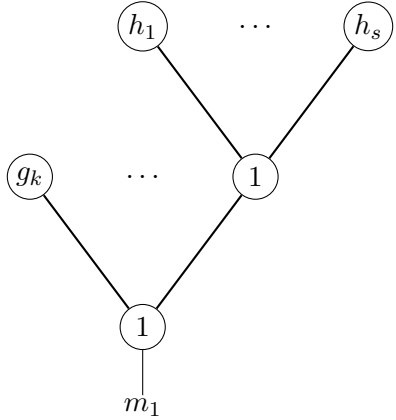
\begin{figure}[H]
  \centering
\begin{tikzpicture}[strata,baseline]
  \node[vertex] (bottom) at (0,0) {$1$};
  \node[vertex] (gk) at (-1.5,2) {$g_k$};
  \node[vertex] (1) at (1.5,2) {$1$};
  \node at ($(gk)!0.5!(1)$) {$\ldots$};
  \node[vertex] (h1) at (0,4) {$h_1$};
  \node[vertex] (hs) at (3,4) {$h_s$};
  \node at ($(h1)!0.5!(hs)$) {$\ldots$};
  \draw[thick] (bottom) -- (gk);
  \draw[thick] (bottom) -- (1);
  \draw[thick] (1) -- (h1);
  \draw[thick] (1) -- (hs);

  \draw (bottom.south) -- ++(0,-0.6)
        node[leglabel,anchor=north] {$m_1$};
\end{tikzpicture}
\end{figure}
As before, we do not yet label the enhancements or mark the zeroes and poles $z_2,\dots,z_n$, but require none of them to be attached to either of the labeled genus~$1$ vertices. We moreover require all marked poles to be attached to the  irreducible component of genus~$h_s$ in the top level. In that case the differential on the mid-level genus~$1$ vertex, and the differential at the bottom vertex are still both going to be genus~$1$ residueless. If both of these genus~$1$ residueless strata contain a connected component of positive genus, let $\alpha_1,\alpha_2$ be some non-zero classes in $\cohH^{1,0}(B_\Lambda^{[-1]},\bC)$ and $\cohH^{0,1}(B_\Lambda^{[-2]},\bC)$. 

Then $1\otimes \alpha_1\otimes \alpha_2$ is a {\em non-tautological} class in $\cohH^{1,1}(B_\Lambda,\bC)$ by using the K\"unneth decomposition of cohomology of $B_\Lambda=B_\Lambda^{[0]}\times B_\Lambda^{[-1]}\times B_\Lambda^{[-2]}$ and \Cref{cor: tautological_kunneth}. If~$\Aut \Lambda=1$, then the same argument as in the proof of \Cref{lem:nontaut_H2_boundary} implies that $\alpha\coloneqq (c_\Lambda)_*(p_\Lambda)^*(1\otimes \alpha_1\otimes \alpha_2)$ is non-trivial in~$\cohH^{1,1}(D_\Lambda)$. We then claim that this class pushes forward to a non-trivial cohomology class on~$\ocH$.

\begin{lem}
If we can arrange a choice of~$g_j,h_j$, of the enhancements, and placement of marked zeroes and poles $z_2,\dots, z_n$, as required, such that the above enhanced level graph~$\Lambda$ is realized in a stratum~$\ocHgnm$, and if furthermore~$\Aut(\Lambda)=1$, then there exists a connected component of~$\ocH\subset\ocHgnm$ such that $\RH^6(\ocH)\subsetneq \cohH^{3,3}(\ocH)$.
\end{lem}
\begin{proof}
Since~$\Lambda$ has no long edge, the group of ghost automorphisms of~$\Lambda$ is trivial by the discussion in \cite[\S~5.3]{chcomo}. Since~$\Aut(\Lambda)=1$ by assumption,~$c_\Gamma^\Gamma$ is again an isomorphism, and we again get a finite morphism $p_\Lambda^\Lambda:D_\Lambda^\circ\rightarrow B_\Lambda^\circ$ on the open stratum. Consequently, arguing the same way as in the proof of \Cref{lem:nontaut_H2_boundary}, we deduce that~$\alpha$ is a non-trivial class in~$\cohH^2(D_\Lambda)$. 

Denote by~$\iota_\Lambda$ the inclusion of the boundary stratum~$D_\Lambda$ into~$\ocH$. We then want to show that $(\iota_\Lambda)_*\alpha\in \cohH^{3,3}(\ocH)\setminus \RH^6(\ocH)$ (and in particular is non-zero).

As before, we look at $(\iota_\Lambda)^*(\iota_\Lambda)_*\alpha$, which is equal to $c_2(N_{D_\Lambda})\cdot\alpha$ by \cite[Lem.~6.1]{calapa}. We will check that $(p_\Lambda)_*(c_\Lambda)^*(\iota_\Lambda)^*(\iota_\Lambda)_*\alpha$ is non-tautological in $B_\Lambda$. Suppose $D_\Lambda\subset \left(D_{\Gamma_1}\cap D_{\Gamma_2}\right)$ for two vertical divisors~$D_{\Gamma_1}$ and~$D_{\Gamma_2}$ (with~$\Lambda$ obtained by degeneration of the top level of~$\Gamma_1$ and of the bottom level of~$\Gamma_2$). Since the intersection of these divisors is smooth by \cite[\S~5]{comoza}, it follows that $c_2(N_{D_\Lambda})= c_1(N_1)\cdot c_1(N_2)$, where $N_i=N_{D_\Lambda/D_{\Gamma_i}}$. The graphs~$\Gamma_i$ are as follows:
\begin{figure}[H]
    \centering
    \begin{subfigure}{0.4\textwidth}
        \centering
        \begin{tikzpicture}[strata,baseline]
  \node[vertex] (bottom) at (0,0) {$1$};
  \node[vertex] (gk) at (-1.5,2) {$g_k$};
  \node[vertex] (H) at ( 1.5,2) {$H$};
  \node at ($(gk)!0.5!(H)$) {$\ldots$};
  \draw[thick] (bottom) -- (gk);
  \draw[thick] (bottom) -- (H);
  \draw (bottom.south) -- ++(0,-0.6)
        node[leglabel,anchor=north] {$z_1$};
\end{tikzpicture}
\caption{$\Gamma_1$}
    \end{subfigure}
    \begin{subfigure}{0.4\textwidth}
    \centering
        \begin{tikzpicture}[strata,baseline]
  \node[vertex] (G) at (0,0) {$G$};
  \node[vertex] (h1) at (-1.5,2) {$h_1$};
  \node[vertex] (hk) at ( 1.5,2) {$h_k$};
  \node at ($(h1)!0.5!(hk)$) {$\ldots$};
  \draw[thick] (G) -- (h1);
  \draw[thick] (G) -- (hk);
  \draw (G.south) -- ++(0,-0.6)
        node[leglabel,anchor=north] {$z_1$};
\end{tikzpicture}
\caption{$\Gamma_2$}
    \end{subfigure}
    \caption{Divisors containing~$D_\Lambda$}
    \label{fig:divisors_containing_DLambda}
\end{figure}
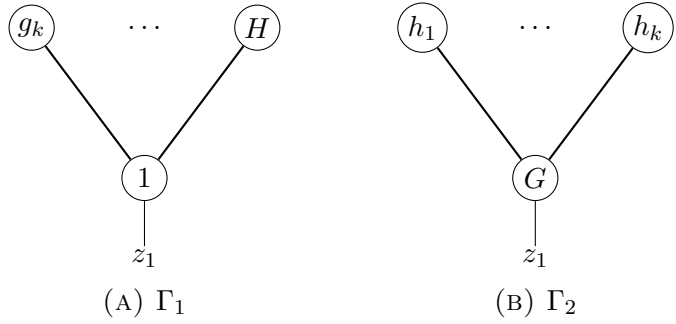
\noindent where $G=g_1+\ldots+g_k+2$
and $H=h_1+\ldots+h_s+1$.

To compute~$c_1(N_i)$ we choose an edge~$e_1$ connecting a vertex in level~$0$ of~$\Lambda$ to the level~$-1$ and an edge~$e_2$ connecting a vertex in level~$-1$ of~$\Lambda$ to the level~$-2$, and use \eqref{eq:general_normal_bundle}:
\[c_1(N_{1})=-\frac{\kappa_{e_1}}{\ell_{\Lambda,1}}(\psi_{e_1^+}+\psi_{e_1^-})-\frac{1}{\ell_{\Lambda,1}}\sum_{\Delta\in \LG_{3,e_1}^\Lambda}\sum_{D_\Delta}\ell_{\Delta,a_{\Delta,\Lambda}}D_\Delta,\]
\[c_1(N_{2})=-\frac{\kappa_{e_2}}{\ell_{\Lambda,2}}(\psi_{e_2^+}+\psi_{e_2^-})-\frac{1}{\ell_{\Lambda,2}}\sum_{\Delta\in \LG_{3,e_2}^\Lambda}\sum_{D_\Delta}\ell_{\Delta,a_{\Delta,\Lambda}}D_\Delta.\]
In particular, $(c_\Lambda)^*(c_1(N_i))$ is the pullback of the correspondingly defined divisor class~$\nu_i=\nu_i^{[-i+1]}+\nu_i^{[-i]}$ on~$B_\Lambda$, as mentioned in the discussion following \eqref{eq:normal_bundle}, and thus $p_\Lambda^*(\nu_1\cdot \nu_2) = c_\Lambda^* c_2(N_{D_\Lambda})$. An application of the projection formula then yields
\[(p_\Lambda)_*(c_\Lambda)^*(\iota_\Lambda)^*(\iota_\Lambda)_*\alpha=\nu_1\cdot\nu_2\cdot (p_\Lambda)_*(c_\Lambda)^*\alpha\,.\] 
By \eqref{eq:cpp} we obtain on the open stratum
$$
 (p_\Lambda)_*(c_\Lambda)^*(\alpha|_{B_\Lambda^\circ})=(\deg p_\Lambda^\Lambda) (\deg q_\Lambda)^2(1\otimes \alpha^{[-1]}\otimes\alpha^{[-2]})|_{B_\Lambda^\circ}\,,
$$ 
and thus on the closed stratum~$B_\Lambda$ the class $(p_\Lambda)_*(c_\Lambda)^*\alpha$ differs from $1\otimes \alpha^{[-1]}\otimes\alpha^{[-2]}$ by some linear combination of boundary divisors of~$B_\Lambda$. So, by \Cref{cor: tautological_kunneth}, to show that $\nu_1\cdot\nu_2\cdot(p_\Lambda)_*(c_\Lambda)^*\alpha$ is non-tautological, it is enough to show that $\nu_1\cdot\nu_2\cdot (1\otimes \alpha^{[-1]}\otimes\alpha^{[-2]})$ is non-trivial. Then, noting $\nu_2^{[-2]}\cdot \alpha^{[-2]}=0$ (since the bottom level is real 2-dimensional), we see that $\nu_1\cdot\nu_2\cdot (1\otimes \alpha^{[-1]}\otimes\alpha^{[-2]})$ is equal to
  \begin{align*}
    &\left(\nu_1^{[0]}+\nu_1^{[-1]}\right)\cdot \left(\nu_2^{[-1]}+\nu_2^{[-2]}\right)\cdot (1\otimes \alpha^{[-1]}\otimes\alpha^{[-2]})\\
    &=\left(\nu_1^{[0]}\otimes\nu_2^{[-1]}\cdot\alpha^{[-1]}\otimes\alpha^{[-2]}\right)+\left(1\otimes\nu_1^{[-1]}\cdot\nu_2^{[-1]}\cdot \alpha^{[-1]}\otimes\alpha^{[-2]}\right).
  \end{align*}
These two summands lie in 
$$\begin{aligned}\cohH^{1,1}&(B_\Lambda^{[0]})\otimes \cohH^{1,2}(B_\Lambda^{[-1]})\otimes \cohH^{1,0}(B_\Lambda^{[-2]})\quad\hbox{and}\\ \cohH^{0,0}&(B_\Lambda^{[0]})\otimes \cohH^{2,3}(B_\Lambda^{[-1]})\otimes \cohH^{1,0}(B_\Lambda^{[-2]})\,,
\end{aligned}$$
respectively, so it is enough to show that~$\nu_1^{[0]}\neq 0$ and $\nu_2^{[-1]}\cdot\alpha^{[-1]}\neq 0$, which would then imply that the first summand is non-zero. This is clear since~$\LG_{3,e_i}^\Lambda$ is non-empty (we can degenerate~$\Lambda$ keeping the underlying dual graph fixed, by simply moving the vertex containing the half edge~$e_i^+$ up), so $\nu_i^{[-i+1]}=0$ would imply~$\psi_{e_i}^+$ is anti-effective, which cannot happen since $\psi$-classes are nef.

The fact that $\nu_2^{[-1]}\cdot\alpha^{[-1]}\neq 0$ follows from \Cref{leray_hirsch}, by checking that its restriction to the open stratum is non-zero. Indeed,~$\alpha^{[-1]}$ restricts non-trivially to the interior~$\cR$ of the residueless stratum~$\overline{\cR}$ associated to the genus~$1$ vertex in level~$-1$, and thus to $$Y\coloneqq f(B_\Lambda^{[-1]})\cap\left(\prod_{v\in \Lambda^{(-1)}}\cH_v\right)=\cR\times \prod_v\cH_v$$ (note that only the genus~$1$ vertex of level~$-1$ is constrained by a residue condition), where~$f$ is the morphism $B_\Lambda^{[-1]}\rightarrow\prod_{v\in \Lambda^{(-1)}}\ocH_v$. Denote $(B_\Lambda^{[-1]})^\circ$ the preimage of~$Y$ in~$B_\Lambda^{[-1]}$ under the map~$f$. By \Cref{leray_hirsch}, one direct summand of $\cohH^3\left((B_\Lambda^{[-1]})^\circ\right)$ is equal to $\cohH^1(F_{-1})\otimes \cohH^2(Y)$, where~$F_{-1}$ denotes the homology fiber of~$f$, as in \Cref{leray_hirsch}. Then by \Cref{leray_hirsch}, the projection of $$\left.\left(\nu_2^{[-1]}\cdot\alpha^{[-1]}\right)\right|_{\left(B_\Lambda^{[-1]}\right)^\circ}\in \cohH^3\left((B_\Lambda^{[-1]})^\circ\right)$$ to $\cohH^1(F_{-1})\otimes \cohH^2(Y)$ is non-trivial. Note that~$\nu_2^{[-1]}$ restricts non-trivially to the fiber~$F_{-1}$ because if~$C\subset F_{-1}$ is a general moving curve in~$F_{-1}$, then $C\cdot D_\Delta>0$ for a degeneration~$\Delta$ of level~$-1$ of~$\Lambda$ (while keeping the underlying dual graph fixed) that makes the edge~$e_2$ longer (i.e. $\Delta\in \LG_{3,e_2}^\Lambda$), so we must have $\nu_2^{[-1]}\cdot C<0$.
\end{proof}
We now finish by checking when such choices are possible, so that this enhanced level graph~$\Lambda$ can be realized in the boundary of a given stratum of differentials.

\begin{proof}[Proof of \Cref{thm:H6}]
As in the previous proofs, while our construction is more general, allowing multiple top and mid-level vertices, for the existence part of the theorem, it is enough to consider the case when $s=k=2$ and $g_2=h_1=1$, while $h_2=g-4$, with all marked zeroes and poles $z_2,\dots,z_n$ placed at this higher genus top level vertex --- by the conditions on~$g$ and~$n$ this is the dual graph of  a stable curve. Then the bottom level residueless stratum is $\cR_{1,3}(m_1,-2,2-m_1)$ as in the previous divisorial boundary constructions, while the mid-level genus~$1$ residueless stratum is $\cR_{1,3}(m_1-4,-2,6-m_1)$. For $m_1\ge 13$ both of these residueless strata have a connected component of positive genus by \Cref{prop:genusLT}, and manifestly the automorphism group of the 3-level graph~$\Lambda$ is trivial. For the proof of the exponential asymptotic growth order we again can replace the middle or top level univalent vertices of genera $g_2=1$ and $h_1=1$ with collections of univalent vertices of distinct genera, yielding still higher asymptotic growth order. The argument for the linear independence of the classes thus constructed is the same as in the previous cases, as they come from disjoint boundary strata.
\end{proof}

\section{Meromorphic strata with $\cohH^2(\cHgnm)\ne 0$}\label{section: open_strata}
All previous results utilized the presence of non-trivial cohomology of some boundary strata to construct non-trivial and non-tautological cohomology classes on~$\ocHgnm$, by push-pull of non-trivial homology classes under the inclusion map of the boundary strata. As these manifestly come from the boundary, they do not yield anything on the open strata~$\cHgnm$ themselves, and here we briefly venture into this question. A systematic study of the boundary complex of~$\cHgnm\subset\ocHgnm$, including a detailed combinatorial definition of the associated enhanced level graph complex, will be presented elsewhere, while our results are mostly in a different range of degrees.

For a fixed connected component~$\ocH$ of~$\ocHgnm$, denote~$D^{[p]}$ the locus of points of multiplicity at least~$p$ in the boundary $D\coloneqq\ocH\setminus\cH$, and denote by~$\widetilde{D}^{[p]}$ its normalization. In particular, $D^{[0]}=\widetilde{D}^{[0]}=\ocH$ and~$D^{[1]}=D$, while in general $D^{[p]}=\cup_{\Gamma\in \LG_p}D_\Gamma$ and $\widetilde{D}^{[p]}=\sqcup_{\Gamma\in \LG_p}\sqcup_{D_\Gamma}\widetilde{D}_\Gamma$, where as usual,~$D_\Gamma$ is an irreducible component of the boundary stratum associated to level graph~$\Gamma$, $\widetilde{D}_\Gamma$ is the normalization of~$D_\Gamma$, and $\LG_p=\cup_{H+L=p} \LG_{L,H}$ denotes the collection of all enhanced level graphs that give codimension~$p$ strata in~$\ocH$. Recall that in our shorthand notation~$D_\Gamma$ 
tacitly stand for the various irreducible connected components of the boundary stratum with a given enhanced level graph.

By Deligne's theorem (see \cite[Ch.~19, Thm.~5.8]{acg2} for details), since~$\ocH$ is a normal crossing compactification of~$\cH$, there exists a spectral sequence abutting to~$\cohH^*(\cH)$ and with~$E_2=E_\infty$ such that:
\begin{equation}\label{eq:ss}
  E_1^{-p,q} =
    \begin{cases}
      \bigoplus_{\Gamma\in \LG_p}\bigoplus_{D_\Gamma}\cohH^{q-2p}(\widetilde{D}_\Gamma,\epsilon^p) & \text{for}\ p>0,\\
      \cohH^q(\ocH) & \text{for}\ p=0,\\
      0 & \text{for}\ p<0,
    \end{cases}       
  \end{equation}
Here~$\epsilon^p$ is a rank~$1$~$\bQ$-local system on~$\widetilde{D}^{[p]}$, which is trivial on~$\widetilde{D}_\Gamma$ if the boundary of~$\ocH$ has simple normal crossings along~$D_\Gamma$. In particular, since~$\ocH$ has simple normal crossing boundary away from the horizontal boundary divisors, it follows that~$\epsilon|_{\widetilde D_\Gamma}$ is trivial for any $\Gamma\in \LG_{*,\le 1}$, i.e.~for any level graph with at most one horizontal edge. 

For the mixed Hodge structure on~$\cohH^*(\cH)$, the weight~$m$ graded piece $\Gr_m^W\, \cohH^k(\cH)$ is then isomorphic to~$E_2^{-m+k,m}$. Since the spectral sequence \eqref{eq:ss} degenerates on the~$E_2$~page, it follows that
\begin{equation}
  \cohH^1(\cH)=E_2^{0,1}\oplus E_2^{-1,2},
\end{equation}
where 
\[
  E_2^{0,1}=\ker(E_1^{0,1}\rightarrow E_1^{1,1})/\im\,(E_1^{-1,1}\rightarrow E_1^{0,1})
\]
and 
\[
  E_2^{-1,2}=\ker(E_1^{-1,2}\rightarrow E_1^{0,2})/\im\,(E_1^{-2,2}\rightarrow E_1^{-1,2}).
\]
Since $E_1^{1,1}=E_1^{-1,1}=E_1^{-2,2}=0$, we obtain
\begin{equation}\label{eq:H2}
  \cohH^1(\cH)=\cohH^1(\ocH)\oplus \ker(\cohH^0(\widetilde{D})\rightarrow \cohH^2(\ocH)).
\end{equation}
In particular, the linear independence of the irreducible components of the boundary of a holomorphic stratum~$\ocH$, given by \Cref{prop:indep},  which we will prove in the following section, yields the vanishing of the kernel in \eqref{eq:H2}, and we thus obtain
\begin{cor}
 For any holomorphic stratum, $\cohH^1(\cHgnm)\cong \cohH^1(\ocHgnm)$.
\end{cor}
We will now see that \Cref{prop:indep} fails already for some genus~$1$ meromorphic strata, which in particular implies that in those cases $\cohH^1(\ocHgnm)\ne \cohH^1(\cHgnm)$.
\begin{rem}\label{rem:no_exceptional_genus1}
Suppose $g=1$ and $\mu=(m_1,\ldots,m_n)\ne (0,\dots,0)$, with $n\ge 3$ (the case of $n=2$ is classical). Then forgetting the first marked point yields a generically finite morphism $f:\ocH_{1,n}(\mu)\rightarrow \overline{\cM}_{1,n-1}$ of degree~$m_1^2$. For any $(E,z_2,\ldots,z_n)\in \cM_{1,n-1}$, the fiber $f^{-1}(E,z_2,\ldots,z_n)$ consists of exactly~$m_1^2$ points $(E,\omega,z_1,\ldots,z_n)$ with~$z_1$ satisfying $m_1z_1=-\sum_{i\ge 2}m_iz_i$ (if such~$z_1$ coincides with any other~$z_i$, we have a nodal curve with a rational tail and the corresponding multi-scale differential). 

Each irreducible boundary divisor of~$\ocH_{1,n}(\mu)$ whose image in~$\overline{\cM}_{1,n}$ under the map~$\pi$ forgetting~$\omega$ is contained in the divisor $\delta_{0,\{1,i\}}\subset\partial\overline{\cM}_{1,n}$ (where the nodal curve contains a rational curve with points~$z_1$ and~$z_i$, attached at a point~$q$ to the elliptic curve~$E$ containing all other marked points), is not contracted under the map~$f$ forgetting~$z_1$. Indeed, the condition $m_1z_1+\sum_{j\ge 2}m_jz_j=0$ on this boundary divisor specializes on~$E$ to $(m_1+m_i)q+\sum_{j\ne 1,i}m_jz_j=0$, while the image under~$f$ is $(E,q,m_2,\dots,\widehat m_i,\dots,m_n)$. Unless $m_1+m_i=0$ and all other~$m_j$ are equal to zero, the image of such a divisor under~$f$ is a hypersurface in~$\cM_{1,n-1}$, so such a divisor is not contracted. However, if $\mu=(a,0,\dots,0,-a)$ (recall that $m_1\ge\dots\ge m_n$ so this is the only case), then~$\ocH_{1,n}(\mu)$ simply does not have such a boundary divisor where~$z_1$ and~$z_n$ collide, by the global residue condition.

Thus the images of all exceptional divisors of~$f$ are contained in the boundary of~$\overline\cM_{1,n-1}$. This fails in general for $g\ge 2$: e.g.~the map~$\ocH_{3,3}(2,1,1)\rightarrow \overline\cM_{3,1}$ defined by forgetting the two simple zeros has an exceptional divisor whose image is the codimension 2 locus in~$\overline\cM_{3,1}$ parameterizing~$(C,z_1)$, where~$C$ is hyperelliptic and~$z_1$ is a Weierstra{\ss} point.
\end{rem}
We use this remark to construct genus~$1$ strata where \Cref{prop:indep} fails.
\begin{rem}
Forgetting the simple zero induces an open immersion $$\cH_{1,3}(a,1,-a-1)\hookrightarrow \cM_{1,2},$$ whose image is the complement of the union $$\pi\left(\cH_{1,2}(a+1,-a-1)\right)\cup\pi\left(\cH_{1,2}(a,-a)\right),$$ where~$\pi$ still denotes the forgetful map $\pi:\cHgnm\to\cM_{g,n}$. Indeed, the only points missed when forgetting the simple zero are where this zero collided either with the other zero or with the pole.

Recall that~$\cohH^2(\overline\cM_{1,2})$ is generated by boundary divisors (see e.g. \cite[Ch.~19, Thm.~4.1]{acg2}). We thus write expressions for these two $\pi$-images as linear combinations of boundary divisors of~$\overline\cM_{1,2}$. The boundary divisors of $\ocH_{1,3}(a,1,-a-1)$ corresponding to these collisions of the simple zero are isomorphic to $\ocH_{1,2}(a+1,-a-1)$ and some connected components of $\ocH_{1,2}(a,-a)$. Pulling back such relations, we obtain expressions for them as linear combinations of some boundary divisors of $\ocH_{1,3}(a,1,-a-1)$ that are pullbacks of boundary divisors of~$\overline\cM_{1,2}$, plus some exceptional divisors of the map $\pi:\ocH_{1,3}(a,1,-a-1)\to\overline{\cM}_{1,3}$. These are then linear relations among the boundary divisors in~$\ocH$, thereby implying that~$\cohH^1(\cH)\neq 0$, by \eqref{eq:H2}.
\end{rem}
A systematic study of~$\cohH^1(\ocH)$ and~$\cohH^1(\cH)$ is being undertaken in \cite{CGMP-first}. Additionally, Dawei Chen has also told us about examples of higher genus meromorphic strata where the irreducible components of the boundary divisor are not linearly independent. 

\smallskip
We now further examine the spectral sequence \eqref{eq:ss} to deduce some results on~$\cohH^2(\cH)$.
By degeneration on the $E_2$~page, we compute
\begin{equation}\label{eq: second cohomology}
  \cohH^2(\cH)= E_2^{0,2}\oplus E_2^{-1,3}\oplus E_2^{-2,4},
\end{equation}
where 
$$
\begin{aligned}
\Gr_2^W\, \cohH^2(\cH)=E_2^{0,2}\, \ &=\ker(E_1^{0,2}\rightarrow E_1^{1,2})/\im(E_1^{-1,2}\rightarrow E_1^{0,2}),\\
\Gr_3^W\, \cohH^2(\cH)=E_2^{-1,3}&=\ker(E_1^{-1,3}\rightarrow E_1^{0,3})/\im(E_1^{-2,3}\rightarrow E_1^{-1,3}),\\
\Gr_4^W\, \cohH^2(\cH)=E_2^{-2,4}&=\ker(E_1^{-2,4}\rightarrow E_1^{-1,4})/\im(E_1^{-3,4}\rightarrow E_1^{-2,4}).
\end{aligned}
$$
Since $E_1^{1,2}=E_1^{-2,3}=E_1^{-3,4}=0$, these three graded pieces of the mixed Hodge structure on~$\cohH^2(\cH)$ are
\begin{equation}\label{eq:GrH2}
\begin{aligned}
\Gr_2^W\, \cohH^2(\cH)&=\coker(\cohH^0(\widetilde{D})\rightarrow \cohH^2(\ocH))\\
\Gr_3^W\, \cohH^2(\cH)&=\ker(\cohH^1(\widetilde{D})\rightarrow \cohH^3(\ocH))\\
\Gr_4^W\, \cohH^2(\cH)&=\ker(\cohH^0(\widetilde{D}^{[2]},\epsilon^2)\rightarrow \cohH^2(\widetilde{D})).
\end{aligned}
\end{equation}

With this we are now ready to demonstrate that many open genus~$1$ strata have non-trivial~$\cohH^2$.
\begin{proof}[Proof of \Cref{prop:H2mer}]
We first prove that $\Gr_4^W \cohH^2(\cH_{1,n}(\mu))\ne 0$ if $n\ge 3$ and $m_k=\pm 1$ for some~$k$. Then the stratum is connected, and we write~$\cH$ for it. Forgetting the~$k^{\text{th}}$ marked point yields a birational morphism $f:\ocH\rightarrow \overline\cM_{1,n-1}$. Let~$\Gamma_i$ be the level graph with one edge and two vertices, one of genus~$0$, containing~$z_k$ and~$z_i$, and the other of genus~$1$, containing all other marked zeroes and poles. The level structure on this graph is determined by the sign of~$\kappa=m_k+m_i+1$, and we recall that such a boundary divisor does not exist only if $(k,i)=(1,n)$ and $\mu=(a,0,\dots,0,-a)$, but then by assumption $a=1$ and this case does not appear. Then every irreducible component~$D_{\Gamma_i}$ is a boundary divisor such that $f(D_{\Gamma_i})$ intersects the interior~$\cM_{1,n-1}$: in fact $f(D_{\Gamma_i})\cap \cM_{1,n-1}=\pi(\cH(\mu_i))$ for 
$$\mu_i\coloneqq(m_1,\ldots,\widehat m_i,\dots,m_k+m_i,\dots,\widehat m_k,\dots,m_n).$$ 
For the argument that follows, we need at least two indices~$i$ such that $\cH(\mu_i)\ne \emptyset$, so this excludes the case $\mu=\pm(1,1,-2)$. Thus by \Cref{rem:no_exceptional_genus1} there is no $f$-exceptional divisor whose image is contained in $f(D_{\Gamma_i})\cap\cM_{1,n-1}$, and since~$f$ is birational and~$f(D_{\Gamma_i})$ is an irreducible divisor, it follows that as a set $f^{-1}(f(D_{\Gamma_i}))$ is the union of~$D_{\Gamma_i}$ and some boundary divisors of~$\ocH$ that are $f$-exceptional.

Since the cohomology $\cohH^2(\overline\cM_{1,n-1})$ is generated by the boundary divisors, the divisor class of~$f(D_{\Gamma_i})$ can be written as a $\bQ$-linear combi\-nation of boundary divisors in~$\overline\cM_{1,n-1}$ such that the coefficient of $\delta_{irr}=12\lambda$ in this linear combination is non-zero (see e.g.~\cite[Prop. 3.1]{chencoskun}). Pulling back such an expression to~$\ocH$ yields $f^*(f(D_{\Gamma_i})=D_{\Gamma_i}+F=\sum_\Gamma a_\Gamma^i D_\Gamma$, where~$F$ is some linear combination of $f$-exceptional divisors, and the sum on the right is only over those irreducible components of the boundary of~$\ocH$ whose $f$-images in~$\overline\cM_{1,n-1}$ are contained in the boundary~$\partial\overline\cM_{1,n-1}$ (since by \Cref{rem:no_exceptional_genus1} there is no exceptional divisor in~$\ocH$ whose $f$-image intersects the interior~$\cM_{1,n-1})$). Furthermore,  the coefficient of the component of horizontal boundary~$D_h^{irr}$ parameterizing differentials on irreducible rational nodal curves is non-zero in this expression. Additionally, since all~$D_{\Gamma_j}$ are not $f$-exceptional, each coefficient~$a_{\Gamma_j}^i$ is zero. Without loss of generality, assume $k\ne 1,2$. Then the manifestly non-zero (since $D_{\Gamma_2}\cdot D_h^{irr}\ne 0$) class
    \[D_{\Gamma_1}\cdot D_{\Gamma_2}-\sum_\Gamma a_\Gamma^1 D_\Gamma\cdot D_{\Gamma_2}\in \cohH^0(\widetilde{D}^{[2]},\epsilon^2)\]
pushes forward to zero in~$\cohH^2(\widetilde{D})$, and thus by \eqref{eq:GrH2} gives a non-trivial class in $\Gr_4^W\, \cohH^2(\cHgnm)$.

\smallskip
Next, we deal with the component~$\cH$ of rotation number~$1$ for the stratum with $n\ge 3$ and $d\coloneqq\gcd(m_1,\ldots,m_n)\ge 2$. Consider the morphism $h:\cH\rightarrow \cX_1(d)$ to the classical modular curve~$\cX_1(d)$ (i.e.~the primitive component of~$\cH(d,-d)$) given by $$(E,\omega,z_1,\ldots,z_n)\mapsto \left(E,\sum_i \frac{m_i}{d}z_i,\sum_i m_i z_i\right).$$
This map extends to the compactification as $\overline h:\ocH\rightarrow\overline \cX_1(d)$. Indeed, the Abel-Jacobi map extends to the locus where the underlying elliptic curve is of compact type, so that the addition map then simply happens on the genus~$1$ component, forgetting all rational components. Further\-more, the locus in~$\overline\cX_1(d)$ where~$E$ degenerates to a rational nodal curve, or to a cycle of rational curves, is a finite set of points, and the image of the locus in~$\ocH$ where~$E$ is not of compact type has to be contained in this finite set.

For $d\ge 2$, the modular curve $\overline\cX_1(d)$ has at least two distinct boundary points $\alpha_1,\alpha_2$ that map to the boundary point of~$\overline{\cM}_{1,1}$, that is to the class of the rational nodal curve. Then the equality $h^{*}(\alpha_1)=h^*(\alpha_2)$ in~$\cohH^2(\ocH)$ yields a non-trivial relation $\sum_{\Gamma}a_\Gamma D_\Gamma=0$ between the components of the boundary divisor parameterizing curves of non-compact type in~$\ocH$. So if~$D_\Lambda$ is a component of~$\partial \ocH$ parameterizing curves of compact type, then $\sum_\Gamma a_\Gamma D_\Gamma\cdot D_\Lambda\in \cohH^0(\widetilde{D}^{[2]},\epsilon^2)$  pushes forward to zero in~$\cohH^2(\widetilde{D})$ and thus gives a non-trivial class in $\Gr_4^W\, \cohH^2(\ocH)$.

\smallskip
Finally, we prove that $\Gr_3^W \cohH^2(\cH_{1,3}(a+1,-1,-a))\ne 0$ for $a=10$ or for $a\ge 12$. Since $a$ and $a+1$ are coprime, the stratum $\cH_{1,3}(a+1,-1,-a)$ is connected, and we write~$\cH$ for it. Then, as before,~$\ocH$ is birational to~$\cM_{1,2}$, by forgetting the simple pole. Thus the coarse moduli space of~$\ocH$ is a rational surface, which implies that $\cohH^1(\ocH)=\cohH^3(\ocH)=0$ (see \cite{kollarfundamentalgroup}). 

To construct a non-zero class in $\Gr_3^W\, \cohH^2(\cH)$, consider the boundary divisor corresponding to the following level graph:
    
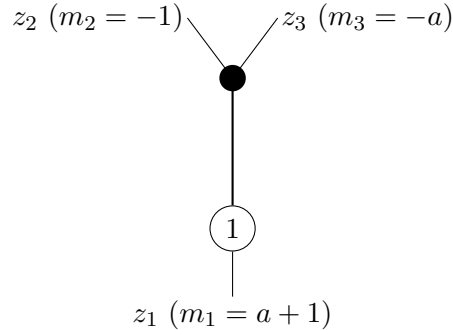
\begin{figure}[H]
  \centering
  \begin{tikzpicture}[strata,baseline]
  \node[vertex] (1) at (0,0) {$1$};
  \fill (0,2) circle (5pt);
  \coordinate (top) at (0,2);
  \draw[thick] (1) -- (top);
  \draw (top) -- ++(-0.6,0.8)
        node[leglabel,anchor=east] {$z_2\ (m_2=-1)$};
  \draw (top) -- ++(0.6,0.8)
        node[leglabel,anchor=west] {$z_3\ (m_3=-a)$};
  \draw (1.south) -- ++(0,-0.6)
        node[leglabel,anchor=north] {$z_1\ (m_1=a+1)$};
\end{tikzpicture}
\end{figure}
\noindent The bottom level vertex parameterizes $\ocH_{1,2}(a+1,-a-1)$. For~$a$ as indicated, the primitive component~$D_\Gamma$ of this  stratum  is the modular curve of level~$a+1$, which has positive genus, so that $\cohH^1(D_\Gamma)\neq 0$. Since~$\cohH^3(\ocH)=0$, it follows from \eqref{eq:GrH2} that $\Gr_3^W \,\cohH^2(\cH)\neq 0$.
\end{proof}
We recall that for $\mu=(0,\dots,0)$ the stratum $\cH_{1,n}(0,\dots,0)=\cM_{1,n}$ has trivial second cohomology \cite[Ch.~19, Thm.~5.1]{acg2}. 

However, the argument above fails for strictly holomorphic strata (of which there are none in genus~$1$). In fact, we have already seen in \S 4 that for a divisor~$D_\Gamma$ with two levels, $\cohH^1(B_\Gamma^\bot)$ usually maps injectively into~$\cohH^3(\ocH)$.

We will now build upon this understanding of the mixed Hodge structure on the open strata, and use the linear independence of the boundary divisors, which is \Cref{prop:indep} proven in the appendix, to obtain \Cref{prop:H2hol}. That is, we now prove that the weight 4 piece of the second cohomology of holomorphic strata is zero. 

\begin{proof}[Proof of \Cref{prop:H2hol}]
We aim to show that $\Gr_4^W\,\cohH^2(\cH)=0$ for any connected component~$\cH$ of any strictly holomorphic open stratum~$\cHgnm$. Recall  expression \eqref{eq:GrH2} for $\Gr_4^W\,\cohH^2(\cH)$. As discussed above, the local system~$\epsilon^2$ on codimension 2 boundary strata can only be non-trivial on~$D_\Lambda$ if $\Lambda\in \LG_{0,2}$, that is if~$\Lambda$ has one level and two horizontal edges. Since the horizontal boundary divisor~$D_h$ of any strictly holomorphic open stratum is irreducible, any such~$D_\Lambda$ is contained in the self-intersection of the irreducible divisor~$D_h$, and thus~$\epsilon^2$ is {\em non}-zero on any such~$D_\Lambda$. But then $\cohH^0(\widetilde D_\Lambda,\epsilon^2)=0$ for any $\Lambda\in \LG_{0,2}$, and thus we only have to deal with enhanced level graphs that have at most one horizontal edge, whence~$\epsilon^2$ is trivial, and we just drop it from the notation. Suppose then for contradiction that
\[
 \beta\coloneqq \sum_{D_\Lambda:\Lambda\in \LG_{2,0}\cup \LG_{1,1}}a_{D_\Lambda}D_\Lambda
\]
gives a non-zero element of $\Gr_4^W\,\cohH^2(\cH)$, where we are summing over irreducible components of the strata for which the enhanced level graph either has three levels and no horizontal edges, or has two levels and one horizontal edge. We aim to show that if~$\beta$ lies in the kernel of the map to~$\cohH^2(\widetilde D)$, then it is the trivial linear combination.

\textit{Step 1:} If~$\Lambda$ is a (possibly horizontal) degeneration of the top level of a divisorial level graph~$\Gamma$ with unique vertex on the top level then~$a_{D_\Lambda}=0$ since the top level of~$\Gamma$ is a holomorphic stratum. Indeed, given a vertical divisor~$D_\Gamma$, we have $(p_\Gamma)_*(c_\Gamma)^*\beta = (\pi^\top)^*\alpha^\top+(\pi^\bot)^*\alpha^\bot$, where $\pi^\top,\pi^\bot$ are projections of $B_\Gamma$ to $B_\Gamma^\top,B_\Gamma^\bot$, respectively. So if $(p_\Gamma)_*(c_\Gamma)^*\beta=0$ then $(\pi^\top)^*\alpha^\top=0$. Then the linear independence of boundary divisors in~$B_\Gamma^\top$ implies~$a_{D_\Lambda}=0$.

\textit{Step 2:} Next, we deal with $\Lambda\in \LG_{1,1}$, so that~$D_\Lambda$ is then contained in the horizontal boundary divisor~$D_h$, whose normalization is isomorphic to $\ocH_{g,n}(\mu,-1,-1)/S_2$. Note that~$D_\Lambda$ are then irreducible boundary divisors in~$\ocH_{g,n}(\mu,-1,-1)$ such that either the top level has more than two vertices or the simple poles are contained in the bottom level (the case that the top level has unique vertex and also has both simple poles is precluded by Step 1). If the top level of~$\Lambda$ has unique vertex then it is also holomorphic, then an argument analogous to part (2) of the proof of \Cref{prop:indep} implies~$a_{\Lambda,i}=0$. Analogously, if the top level of~$\Lambda$ has more than one vertex, then the argument analogous part (3) of the proof of \Cref{prop:indep} implies~$a_{\Lambda,i}=0$.

\textit{Step 3:} Now we look at $\Lambda\in \LG_{2,0}$ such that~$D_\Lambda$ is contained in a vertical boundary divisor~$D_\Gamma$ where~$\Gamma$ has at least two vertices in the top level. As in the proof of \Cref{prop:indep}, we can proceed by induction on the number of vertices in the top level of~$\Lambda$. If~$\Lambda$ has unique top level vertex, we proceed as in part (2) of the proof of \Cref{prop:indep}, and if~$\Lambda$ has more that one top level vertex, we proceed as in part (3) of the proof.
\end{proof}
\begin{rem}
    Based on \Cref{claim:indepH3}, one might wonder if also the weight three part $\Gr_3^W\,\cohH^2(\cH)$ is zero for any holomorphic stratum~$\cH$. However, we are unable to deduce this for two reasons. First, while $\cohH^1(B_\Gamma^{[-1]})$ injects into~$\cohH^3(\ocH)$ for various divisorial level graphs~$\Gamma$, showing the same for~$\cohH^1(B_\Gamma^{[0]})$ can be tricky (though we do not know an example of an open holomorphic stratum with non-trivial first cohomology). The second difficulty is that since in general $D_\Gamma$ and $B_\Gamma$ are only related by a finite correspondence and may not be birational, it is possible that $\cohH^1(B_\Gamma)\subsetneq \cohH^1(D_\Gamma)$. So even if $\cohH^1(B_\Gamma^{[0]})=0$, our argument is insufficient to conclude that~$\cohH^1(D_\Gamma)$ injects into~$\cohH^3(\ocH)$.
\end{rem}

\section{Proof of \Cref{prop:genusLT}}\label{proof_prop:genusLT}
The proof of this proposition is by obtaining suitable upper bounds for the formulas for Euler characteristics obtained by Lee and Tahar \cite{leetahar}, and evaluating the numbers in a short list of cases. As this is a combinatorial argument and the techniques are disjoint from the rest of the paper, we have moved this argument to a separate section.

We start by recalling Lee and Tahar's explicit formula for the  Euler characteristic of a residueless stratum $\cR=\cR_{1,1+p}(a,-b_1,\dots,-b_p)$, by which they and we always mean the disconnected union of all its connected components. While their main focus is on the orbifold Euler characteristic (and we note that the orbifold structure they use is induced by the natural projective structure and not by the multi-scale compactification), we want the existence of a connected component of positive genus. For this we will need to compare the manifold Euler characteristic of the coarse moduli space (which they compute) to~$2$ times the number of connected components of~$\cR$, computed in \cite{lee}.

The case $p=1$ is classical, being that of modular curves, see for example \cite{mazur2}. The case $p=2$ will be more complicated combinatorially due to some extra orbifold points, which result in a more elaborate formula in \cite{leetahar}. We thus first deal with the case of at least 3 poles, $p\ge 3$.

Indeed, \cite[eq.~(1.2)]{leetahar} \footnote{Everywhere, we are using the corrected version of that paper, which had a correction to an inaccuracy in dealing with cusps in this formula}, for $p\ge 3$ gives
\begin{equation}\label{eq:chiR}
\chi\left(\cR\right)= \frac{p!B}{48}(A-a^2-4a-4)+\varepsilon, 
\end{equation}
where $a=\sum_i b_i$, $A=\sum_i (b_i)^2$ and $B=\prod_i (b_i-1)$, and $\varepsilon=0$ unless $p=3$, and $b_1,b_2,b_3$ are all even, in which case it is $\varepsilon=-1/2$. 

Furthermore, the Euler characteristic of the compactified residueless stratum is given for $p\ge 3$ by \cite[eq.~(1.3)]{leetahar}:
\begin{equation}\label{eq:chioR}
\chi\left(\overline{\cR}\right)= \chi\left(\cR\right)+\frac12 p!B
+\sum_{t=1}^p \sum_{\tau\in S_p}\sum_{C_i=1}^{b_i-1} \frac{1}{2t} \gcd \left( \sum_{i=1}^t C_{\tau(i)},\sum_{i=1}^t b_{\tau(i)} \right)+\varepsilon_0
\end{equation} 
where the (explicitly given) correction term~$\varepsilon_0$ can only be non-zero for some cases of $b_1,b_2,b_3$ for $p=3$, and is then between~$0$ and~$3$. Here the notation $\sum_{C_i=1}^{b_i-1}$ means that the sum over~$C_i$ is taken for each $i=1,\dots,p$, and~$\tau(i)$ denotes the images under the permutation~$\tau$. \footnote{We note that the original version of \cite{leetahar} contained an inaccuracy in the computation of cusps and their stabilizer, so that in particular the original formula gave half-integers for $p=3$ and even $b_1,b_2,b_3$. We are using the final corrected formula, and thank Lee and Tahar for detailed discussions of these issues}

We now observe simply that
$$
\sum_{C_1=1}^{b_1-1}\dots\sum_{C_p=1}^{b_p-1}\gcd \left( \sum_{i=1}^t C_{\tau(i)},\sum_{i=1}^t b_{\tau(i)} \right)\le \sum_{C_1=1}^{b_1-1}\dots\sum_{C_p=1}^{b_p-1} \sum_{i=1}^t C_{\tau(i)}
$$
Change the order of summation by first taking all sums over~$C_j$, and only then summing over~$i$. Then for a fixed permutation $\tau\in S_p$ and for a fixed index~$i$ in the last sum, all $j\in [1,\dots,p]\setminus \tau(i)$ do not matter, and thus just contribute an overall factor of $\prod_{j\ne \tau(i)}(b_j-1)$. For each such choice, the sum over~$C_{\tau(i)}$ gives $b_{\tau(i)}(b_{\tau(i)}-1)/2$, so that altogether for a fixed~$i$ we simply get $b_{\tau(i)}\cdot B/2$. Summing these contributions (still for a fixed~$t$ and fixed~$\tau\in S_p$) over all~$i$ then gives~$B/2$ times $\sum_{i=1}^t b_{\tau(i)}$. Thus, denoting the double sum over~$\tau$ and~$C_i$ for a fixed~$t$ by~$T_t$, we obtain the bound
$$
 T_t\le \frac{B}{4t}\sum_{\tau\in S_p}\sum_{i=1}^t b_{\tau(i)}=\frac{B}{4t}\cdot a\cdot t\cdot (p-1)!=\frac{(p-1)!aB}{4}\,,
$$
where we simply observed that there are~$p!$ different permutations~$\tau\in S_p$, for each of them we are summing~$t$ different~$b_j$'s, and thus each of the~$p$ summands~$b_j$ appears $p!\cdot t/p$ times.

Applying this bound for each of $T_1,\dots,T_p$ finally gives the upper bound (for $p\ge 3$)
\begin{equation}
 \chi\left(\overline{\cR}\right)\le\frac{p! B}{48}(A-a^2+8a+20)+\varepsilon+\varepsilon_0\,.
\end{equation}
Observe now that for a fixed~$a$, the sum of squares $A=\sum b_i^2$ is maximal when one of the~$b_i$ is as large as possible while others are as small as possible. Indeed, to prove this formally suppose~$b_p$ is the maximal one among the~$b_i$, and write $b_p=a-b_1-\dots-b_{p-1}$, so that then $\tfrac{\partial}{\partial b_i}\sum b_i^2=2(b_i-b_p)\le 0$ for all $i=1,\dots,p-1$, and thus decreasing~$b_i$ while increasing~$b_p$ by the same amount increases~$A$ while keeping~$a$ fixed.

Since the residueless strata only exist for each $b_i\ge 2$, this means that for fixed~$a$, and for any $p\ge 3$, the above upper bound for the Euler characteristic is maximal for
\begin{equation}\label{eq:boundp}
\begin{aligned}
 \chi&\left(\overline\cR_{1,1+p}(a,-2,\dots,-2,a-2p+2)\right)\\
 &\le\frac{p!\cdot (a-2p+1)}{48}\left(4p+(a-2p)^2-a^2+8a+20\right)+\varepsilon+\varepsilon_0\\
 &=
 \frac{p!\cdot (c+1)}{48}\left(20+20p-4p^2+c(8-4p)\right)+\varepsilon+\varepsilon_0
\end{aligned}
\end{equation}
where we have denoted $c\coloneqq a-2p\ge 0$. We now determine all the cases when this expression may be non-negative.

For $p\ge 6$, we see that $20+20p-4p^2<0$ and $8-4p<0$, so that this expression is always negative (recall that~$\varepsilon$ may only be non-zero for $p=3$). For $p=5$, $20+20p-4p^2=20$ while $8-4p=-12$, so that one only needs to check the cases with $c\le 1$, that is with $a=10$, in which case one computes $\chi(\overline{\cR}_{1,6}(10,-2,-2,-2,-2,-2))=50$, and $a=11$, in which case already $\chi(\overline{\cR}_{1,6}(11,-2,-2,-2,-2,-3)=-240$.

Similarly, for $p=4$ one has $20+20p-4p^2=36$ while $8-4p=-8$, so that one needs to check the cases of $0\le c\le 5$, and by enumerating them all and computing the values of the Euler characteristic using a computer one sees that $\chi(\overline{\cR}_{1,5}(8,-2,-2,-2,-2))=18$ is the only case with positive Euler characteristic.

Finally, for $p=3$ one takes into account the term $0\le\varepsilon+\varepsilon_0 \le 3$, so that $20+20p-4p^2+\varepsilon\le 47$, and $8-4p=-4$, so that one needs to write a computer script (which we did, in Maple, and in Python for a cross-check) to verify all possible cases for $c\le 12$, finding precisely the $p=3$ cases that appear in \eqref{eq:exceptionsLT}. We will return to the full list of cases for $p\ge 3$ with positive Euler characteristic after dealing with the case of $p=2$.

\smallskip
Finally, the case of $p=2$ is different, in particular because the corresponding residueless genus one strata may be orbifolds in this case. One of the main results of \cite{leetahar} is the formula for their orbifold Euler characteristic given by \cite[eq.~(1,1)]{leetahar}, but the relevant result for us is \cite[eq.~(6.10)]{leetahar} giving the manifold Euler characteristic:
\begin{equation}
    \chi\left(\overline{\cR}_{1,3}(a,-b_1,-b_2)\right)=\frac{p!B}{48}(A-a^2-4a+20)+T_1+T_2 +\varepsilon+\eta
\end{equation}
where $\varepsilon=\varepsilon_0-\varepsilon_1+\varepsilon_2$, $\varepsilon_0\in \{0,2,3\}$ by \cite[eq.~(6.2)]{leetahar}, $\varepsilon_1=\frac{a}2$ if~$b_1$ and~$b_2$ are both even, and zero else, by \cite[eq.~(6.4)]{leetahar}, $\varepsilon_2=\frac{a-1}2$ if~$b_1$ and~$b_2$ are both even, and zero otherwise, by \cite[eq.~(6.6)]{leetahar}, and $\eta\in \{0,\tfrac23\}$ by \cite[eq.~(6.7)]{leetahar}. Thus altogether $\epsilon_2-\epsilon_1\in\{0,-1/2\}$, and the total correction $\varepsilon+\eta$ is bounded within~$[-1/2,3]$. Finally,~$T_1,T_2$ are the same sums as before:
$$
 T_1=\frac12\sum_{\tau\in S_2}\sum_{c_1=1}^{b_1-1}\sum_{c_2=1}^{b_2-1} \gcd(c_{\tau(1)},b_{\tau(1)})
$$
$$
 T_2=\frac14\sum_{\tau\in S_2}\sum_{c_1=1}^{b_1-1}\sum_{c_2=1}^{b_2-1} \gcd(c_1+c_2,b_1+b_2)
$$
As before, by estimating $\gcd(c_1+c_2,b_1+b_2)\le c_1+c_2$, we obtain (as both permutations give the same contribution)
$$
\begin{aligned}
 T_2&\le \frac12 \sum_{c_1=1}^{b_1-1}\sum_{c_2=1}^{b_2-1}(c_1+c_2)\\ &=\frac12(b_2-1)\frac{(b_1-1)b_1}2+\frac12(b_1-1)\frac{(b_2-1)b_2}2=\frac{(b_1+b_2)(b_1-1)(b_2-1)}{4}
\end{aligned}
$$
by summing the two terms separately.

For~$T_1$, we could also estimate the~$\gcd$ similarly, but we prefer to give a sharper upper bound for this case. Indeed, observe that for any integer $N\ge 2$ the following bound holds:
$$
 \sum_{c=1}^{N-1}\gcd(c,N)\le \sum_{d|N,\, d<N} d\cdot \left(\frac{N}{d}-1\right)<\sum_{d|N,\, d<N} d\cdot \frac{N-1}d
 \le 2(N-1)\sqrt{N}
$$
where we simply noted that for any~$d$ there are $N/d-1$ integers between~$0$ and~$N-1$ that are divisible by~$d$ (and that bound is overcounting due to those integers that are divisible by a multiple of~$d$); we then observed that half of all divisors~$d$ of~$N$ are smaller than or equal to~$\sqrt{N}$. Of course this bound is not optimal (at least the factor of 2 can be removed), but it suffices for our purposes as we thus obtain 
$$
    T_1=\frac12\sum_{c_1=1}^{b_1-1}\sum_{c_2=1}^{b_2-1} \left(\gcd(c_1,b_1)+\gcd(c_2,b_2)\right)
    \le (b_1-1)(b_2-1)(\sqrt{b_1}+\sqrt{b_2})
$$
so that altogether we obtain the bound
\begin{equation}\label{eq:bdp=2}
\begin{aligned}
    \chi&\left(\overline{\cR}_{1,3}(a,-b_1,-b_2)\right)\le \\ -&\frac{(b_1-1)(b_2-1)}{12}(b_1 b_2+2b_1+2b_2+2)+\frac{(b_1+b_2)(b_1-1)(b_2-1)}4\\ &+(b_1-1)(b_2-1)(\sqrt{b_1}+\sqrt{b_2})+\varepsilon+\eta\,\\
    &=B\left(-\frac{b_1 b_2+2b_1+2b_2+2}{12}+\frac{b_1+b_2}{4}+\sqrt{b_1}+\sqrt{b_2} \right)+\varepsilon+\eta\\
    &=B\left(-\frac{B+1}{12}+\sqrt{b_1}+\sqrt{b_2}\right)+\varepsilon+\eta
\end{aligned}
\end{equation}
We now observe that for~$B$ fixed, the value of $\sqrt{b_1}+\sqrt{b_2}$ is maximized when $b_1=2$ (by expressing~$b_2$ in terms of~$b_1$ and~$B$ and computing the derivative, and recalling that in our convention $2\le b_1\le b_2$). Thus setting $b_1=2$ we obtain
$$
\chi\left(\overline{\cR}_{1,3}(a,-b_1,-b_2)\right)\le B\left(-\frac{B+1}{12}+\sqrt{2}+\sqrt{B+1}\right)+\varepsilon+\eta
$$
Observing that certainly this expression is negative for large~$B$ and computing the values of the main term (without the~$\varepsilon+\eta$), we see that for $B\ge 176$ the main term is~$\le -5$, so that the entire expression is negative. Thus it remains to find all values of~$b_1,b_2$ with $B=(b_1-1)(b_2-1)<176$, which we have implemented using Maple (and also in Python for a cross-check). All the cases where the Euler characteristic is strictly positive, for~$p=2$, are given in the following table. In the table we also give the number of connected components of these residueless strata as determined by \cite{lee}, as the sum of the numbers of non-hyperelliptic and hyperelliptic components. We recall that non-hyperelliptic components are classified by the rotation number~$r|\gcd(b_1,b_2)$, except that no non-hyperelliptic components exist for $(4,-2,-2)$ and that there is no non-hyperelliptic component of rotation number~$3$ for $(6,-3-3)$ nor of rotation number~$4$ for $(8,-4,-4)$ by \cite[Thm.~1.7]{lee}. The number of hyperelliptic components is equal to the number of ramification profiles, by \cite[Thm.~1.4]{lee}. In our case of one zero and $p=2$, there are in fact~$2$ hyperelliptic components if $b_1=b_2$ is even, so that the ramification profile can either fix or switch the poles; and there is a unique hyperelliptic component if $b_1<b_2$ is even, so that the ramification profile fixes both poles, or if $b_1=b_2$ is odd, so that the ramification profile swaps these poles.

We put a checkmark in the ``Ok" column if 
$$\chi\left(\overline{\cR}_{1,3}(b_1+b_2,-b_1,-b_2)\right)<2 h^0\left(\overline{\cR}_{1,3}(b_1+b_2,-b_1,-b_2)\right),$$ 
which indicates that there exists a connected component of this residueless stratum that has positive genus. 
$$
\begin{array}{|c|c|c|c|}
\hline
b_1+b_2,-b_1,-b_2&\chi\left(\overline{\cR}_{1,3}\right)&h^0\left(\overline{\cR}_{1,3}\right)&\hbox{Ok?}\\
\hline
4,-2,-2&4&2=0+2&\\
5,-2,-3&2&1=1+0&\\
6,-2,-4&6&3=2+1&\\
7,-2,-5&2&1=1+0&\\
8,-2,-6&6&3=2+1&\\
10,-2,-8&4&3=2+1&\checkmark\\
12,-2,-10&2&3=2+1&\checkmark\\
6,-3,-3&6&3=2+1&\star\\
7,-3,-4&2&1=1+0&\\
8,-3,-5&2&1=1+0&\\
9,-3,-6&2&2=2+0&\checkmark\\
8,-4,-4&8&4=2+2&\\
\hline
\end{array}
$$
Here the~$\star$ for $\overline{\cR}_{1,3}(6,-3,-3)$ indicates that as we were told by Myeongjae Lee there is a missing case in \cite{lee} of two different non-hyperelliptic components of rotation number $r=1$, in addition to the hyperelliptic component. This missing case is parallel but much simpler than the case of $\overline{\cR}_{1,5}(12,-3,-3,-3,-3)$ treated in \cite{leetahar}; recall furthermore that there is no component of rotation number $r=3$ in $\cR_{1,3}(6,-3,-3)$. In all the cases with~$\checkmark$ we are guaranteed to have at least one connected component of positive genus, which leaves the remaining $p=2$ cases as listed in the statement of the proposition.

\smallskip
We now return to the cases with $p\ge 3$ where the Euler characteristic was positive (the full list for $p=3$ and $p=4$ was found by a short computer evaluation of the formulas, as discussed above), to check if it is equal to double the number of connected components in each of these cases. 

For $\cR(6,-2,-2,-2)$ there exist no non-hyperelliptic components. Every ramification profile has to fix the unique zero. Thus there is one ramification profile that fixes all poles, and 3 more ramification profiles that swap one pair of poles, and the total number of connected components is thus equal to 4, and the Euler characteristic is computed to be 8. Thus all 4 connected components have genus zero.

For $\chi\left(\overline\cR(7,-2,-2,-3)\right)=2$ we notice that there there is a unique non-hyperelliptic component, and no hyperelliptic components (an odd order zero cannot be fixed). Thus the unique connected component is rational.

For $\chi\left(\overline\cR(8,-2,-3,-3)\right)=4$ there is a unique non-hyperelliptic component, and a unique ramification profiles swapping the two third order poles. Thus both connected components are rational.

For $\chi\left(\overline\cR(8,-2,-2,-4)\right)=8$ there are two non-hyperelliptic components with $r=1$ and $r=2$, and two hyperelliptic components depending on whether the ramification profile fixes or switches the double poles, and thus each of these 4 connected components is rational.

For $\chi\left(\overline\cR(8,-2,-2,-2,-2)\right)=18$, there are no non-hyperelliptic components. There are~$\binom{4}{2}=6$ hyperelliptic connected components where one pair of poles is swapped, and the remaining two poles and the zero are fixed, and 3 ramification profiles where two pairs of poles are swapped. Thus each of the 9 connected components is rational.

For $\chi\left(\overline\cR(10,-2,-2,-2,-2,-2)\right)=50$, there are no non-hyperelliptic components. There are~$\binom{5}{2}=10$ ramification profiles where one pair of poles is swapped, and the zero and the remaining 3 poles are fixes, and $15=5\cdot 3$ ramification profiles where the zero and one pole are fixed, and the remaining 4 poles are swapped in a choice of one of 3 partitions into pairs. Thus all connected components are rational.

\appendix
\section{Linear independence and extremality of irreducible components of boundary divisors}\label[appendix]{lin_indep}
In this appendix we prove the linear independence of irreducible component of boundary divisors in the holomorphic strata, as well as the extremality of irreducible components of the vertical boundary divisors, that is, we prove \Cref{prop:indep} and \Cref{prop:extreme}. These results were already known to some of us, and implicit in the computations done in \cite{comoza} and \cite{chennonvarying}. However, they have not appeared explicitly in the literature, and thus we give a complete proof for future use.
\begin{proof}[Proof of \Cref{prop:indep}]
First note that on a holomorphic stratum, any graph in~$\LG_{0,1}$ (i.e.~with only one level and one horizontal edge) must be a loop, as otherwise the residue theorem would be violated. Matching the classification of connected components of~$\cHgnm$ for a holomorphic~$\mu$, determined by \cite{kozo}, with the connected components of $\cH_{g-1,n}(\mu,-1,-1)$, determined by \cite{boissy}, then shows that the boundary stratum~$D_h$ corresponding to the loop dual graph is irreducible (This is the proof given for Claim B in \cite{dogr}, which we thus recalled).

Suppose now that we have a linear relation among the boundary divisors of~$\ocH$
\begin{align}\label{equation: boundary relation}
  a_hD_h+\sum_{\Gamma^+\in \LG_{1,0}}\sum_{D_\Gamma}a_{D_\Gamma}D_\Gamma=0,
\end{align}
where we sum over all graphs~$\Gamma^+$ with two levels and no horizontal edges, and over all irreducible components~$D_\Gamma$ of the corresponding  boundary divisor, and $a_h,a_{D_\Gamma}\in\bQ$. We will show that each coefficient in \eqref{equation: boundary relation} has to be zero, by intersecting with various curve classes.

\subsubsection*{Horizontal boundary divisors} 
Let~$C_0$ be the closure in~$\ocH$ of a Teichm\"ul\-ler curve. It was proven in \cite[\S~3]{chenmoller_sumlyap} that a Teichm\"uller curve is never compact and its closure~$C_0$ only intersects the boundary at the points contained in~$D_h$, and not in any vertical boundary divisor. Consequently, $D_\Gamma\cdot C_0=0$ for all irreducible components of the vertical boundary, whereas $D_h\cdot C_0\neq 0$. In particular, from \eqref{equation: boundary relation}, we see that~$a_h=0$.

\subsubsection*{Vertical boundary divisors with unique vertex at the top level} Suppose $\Gamma^+\in \LG_{1,0}$ has a unique vertex~$v$ at level~$0$. Following the notation of \cite{comoza}, for a given~$D_\Gamma$ we write $B_\Gamma=B_\Gamma^\top\times B_\Gamma^\bot$ for the decomposition up to commensurability in the levelwise generalized strata, where in this case~$B_\Gamma^\top$ is itself the connected component~$\cH'$ of a stratum~$\cH_{g',n'}(\mu')$ of holomorphic differentials on the (connected) Riemann surface~$X_v$.

If $g(X_v)\ge 2$, then $B_\Gamma^\top=\cH'$ contains a Teichm\"uller curve~$C$, which again only intersects non-trivially the horizontal boundary divisor $D_h'\subset\cH'$. We denote by $C_\Gamma\coloneqq c_\Gamma\circ p_\Gamma^{-1}(C)\subset D_\Gamma$ the corresponding curve in the boundary divisor. It follows that~$C_\Gamma$ only intersects the boundary divisors~$D_h$ and~$D_\Gamma$ non-trivially. (Here we use again that by the results of \cite[\S~5]{comoza}, the different irreducible components~$D_\Gamma$ of the boundary divisor with a fixed~$\Gamma^+$ are pairwise disjoint.)

The normal bundle of~$D_\Gamma$ is given by \eqref{eq:normal_bundle}. This formula implies 
$$c_1(N_\Gamma)\cdot C_\Gamma = -\frac{\kappa_e}{\ell_\Gamma}\psi_{e^+}\cdot C_\Gamma$$
since~$C_\Gamma$ intersects vertical boundary divisors trivially, and since more\-over
$\psi_{e^-}\cdot C_\Gamma=0$ because~$C_\Gamma$ is contained in~$B_\Gamma^\top$ whereas~$\psi_{e^-}$ is a class pulled back from~$B_\Gamma^\bot$. Then by \cite[\S 6, eq. (39)]{chcomo}, modulo the vertical boundary divisors in~$B_\Gamma^\top$, the intersection number $\psi_{e^+}\cdot C_\Gamma$ on the right hand side is equal to a positive constant times $(12\lambda-D_h')\cdot C_\Gamma$. As computed in \cite[Prop.~4.5]{chenmoller_sumlyap}, the ratio $D_h'\cdot C_\Gamma/\lambda\cdot C_\Gamma$ is strictly less than 12, and thus $$\psi_{e^+}\cdot C_\Gamma=\text{(const)}\cdot(12\lambda-D_h')\cdot C_\Gamma>0,$$ so that $D_\Gamma\cdot C_\Gamma<0$. Since~$C_\Gamma$ intersects only~$D_h$ (and its intersection with~$D_h$ is the same as with~$D_h'$) and~$D_\Gamma$ at the boundary, from \eqref{equation: boundary relation}, and using the already proven~$a_h=0$, we see that~$a_{D_\Gamma}=0$.

We next deal with the case that $g(X_v)=1$ (notice that since~$\mu'$ is holomorphic, it cannot happen that $g(X_v)=0$). Then~$B_\Gamma^\top$ is a finite quotient of the stratum~$\ocH_{1,n'}(0,\dots,0)$, and thus is simply birational to a finite quotient of~$\overline\cM_{1,n'}$. 

We proceed by induction on~$n'$. If $n'=1$, then $B_\Gamma^\top\cong\overline\cM_{1,1}$ is a curve, and we consider its image as a curve $C_\Gamma\subset D_\Gamma$. Note that~$D_\Gamma$ is contained in the preimage of $\delta_{1,\emptyset}\subset \overline\cM_{g,n}$ in~$\ocHgnm$, and~$C_\Gamma$ pushes forward to a non-trivial curve in~$\overline\cM_{g,n}$ that intersects~$\delta_{1,\emptyset}$ negatively. Consequently, $C_\Gamma\cdot D_{\Gamma,i}<0$. Additionally, the curve~$C_\Gamma$ only intersects~$D_h$ but no other boundary divisor, so, as before, we obtain~$a_{D_\Gamma}=0$ in \eqref{equation: boundary relation}.

For $n'\geq 2$ consider the curve $C_\Gamma\subset D_\Gamma$ obtained by varying a marked point along a fixed elliptic curve in~$\overline\cM_{1,n'}$. By \cite[Ch.~13, Prop.~3.10]{acg2} and \cite[Lem.~6.4]{chcomo}, the  intersection number $\ell_\Gamma\left(\sum_{e\in E(\Gamma)}1/p_e\right)C_\Gamma\cdot D_{\Gamma,i}$ is equal to minus the sum of the degrees of~$\psi$ classes on the underlying curves in~$\overline\cM_{1,n'}$, and thus $C_\Gamma\cdot D_\Gamma<0$. On the other hand, the only boundary divisors that~$C_\Gamma$ intersects positively are some~$D_{\Gamma'}$ where~$\Gamma'$ also has unique top level vertex of genus~$1$ with valence~$n'-1$, and by induction, $a_{D_{\Gamma'}}=0$ for all such irreducible components. Consequently, we see that~$a_{D_\Gamma}$ also has to be~$0$ in \eqref{equation: boundary relation}.

\subsubsection*{Vertical boundary divisors with multiple vertices at the top} 
If~$\Gamma$ is an enhanced level graph with $r\ge 2$ vertices at the top level, then~$D_\Gamma$ is necessarily an exceptional divisor over the Deligne-Mumford compactification. Note that multi-scale differentials in $C\subset B_\Gamma^\top$ can only degenerate the level structure of the top level of~$\Gamma$ while keeping the underlying dual graph fixed. Thus if~$C$ is a generic curve in~$B^\top_\Gamma$ such that $C_\Gamma\coloneqq c_\Gamma\circ p_\Gamma^{-1}(C)$ gets contracted under the morphism $\pi:\ocH\rightarrow\overline\cM_{g,n}$, then~$C_\Gamma$ can only intersect another boundary divisor~$D_{\Gamma'}$ if~$\Gamma'$ has at most~$r-1$ vertices at the top level. Additionally, for any boundary divisor $\delta\subset\overline\cM_{g,n}$ the intersection  number $\pi^*\delta\cdot C_\Gamma=0$, so by \cite[Lem.~6.4]{chcomo}$D_\Gamma\cdot C_\Gamma<0$, and again by induction in~$r$ we conclude that~$a_{D_\Gamma}=0$.

Thus we have shown that the coefficients of all divisors in \eqref{equation: boundary relation} are zero, and thus the only linear relation among them is trivial.
\end{proof}

We remark that the fourth-named author of the  appendix, joint with Matteo Costantini and Jonathan Zachhuber, had an unpublished alternative argument to show the linear independence of the boundary divisors of the holomorphic strata, in a slightly weaker sense that does not distinguish irreducible components of a vertical boundary divisor with the same level graph. Their method was to intersect the boundary divisors with curve classes similar to Kontsevich's $\beta$-class appearing in the computations of dynamical invariants of the strata (see \cite{Kont97, CMSZ-volume}). Heuristically speaking, the $\beta$-class numerically represents the ``limit'' of Teichm\"uller curves. Therefore, this method essentially goes along the same circle of ideas as using Teichm\"uller curves. In order to not  make the paper longer, here we skip the details of this alternative argument.

We now prove the extremality of the vertical boundary divisors in the effective cone of a connected component of a holomorphic stratum. We remark that the study of effective cones and their extremal rays plays an important role for understanding the birational geometry of moduli spaces, see  \cite{CFMEffective, ChenExtremal} for an introduction to this circle of ideas as applied to moduli spaces.

\begin{proof}[Proof of \Cref{prop:extreme}]
We continue to use the same notation as in the proof of \Cref{prop:indep}, so that in particular~$D_\Gamma$ is an irreducible component of the vertical boundary divisor in a connected component~$\ocH$ of a holomorphic stratum~$\ocHgnm$. 

If~$\Gamma$ has a unique vertex at level~$0$, then $D_\Gamma\cdot C_\Gamma < 0$ for all Teichm\"uller curves~$C_\Gamma$ contained in the top level. Indeed, from \cite[Section 4]{chenmoller_sumlyap}, the ratio $(D_\Gamma\cdot C_\Gamma)/(D_\Gamma\cdot \lambda)$ is a negative constant which is determined by~$\Gamma$ and independent of each individual ~$C_\Gamma$. Since the union of Teichm\"uller curves is Zariski dense in every holomorphic stratum, the extremality of~$D_\Gamma$ then follows from the same argument as in \cite[Section 4]{GheorghitaHodge} and \cite[Section 5]{chennonvarying}. 

If~$\Gamma$ has at least two vertices on level~$0$, then~$C_\Gamma$ used in the proof of \Cref{prop:indep} is a moving curve in~$D_\Gamma$, and satisfies $D_\Gamma\cdot C_\Gamma < 0$. Then the extremality of~$D_\Gamma$ follows from \cite[Lemma 4.1]{chencoskun}. Alternatively, one observes that~$D_\Gamma$ is an irreducible exceptional divisor under the birational contraction from~$\ocH$ to its image in~$\overline{\cM}_{g,n}$, which also implies the extremality of~$D_\Gamma$ (see \cite[Section 1.5]{RullaThesis}). 
\end{proof}

Finally, we remark that the above proof of extremality does not apply to the (irreducible) horizontal boundary divisor~$D_h$ in a holomorphic stratum. The issue is that a generic element in~$D_h$ corresponds to a differential in the meromorphic stratum $\cH(\mu, -1, -1)$, where the two simple poles occur at the two branches of the horizontal node after normalization. However, analogues of Teichm\"uller curves in strata of meromorphic differentials are not known in general. Thus determining whether~$D_h$ is extremal in the pseudo-effective cone of a connected component of any holomorphic stratum remains an interesting open question. Determining which irreducible boundary divisors of which meromorphic strata are extremal is also an interesting open question.

\bibliography{contraction}
\bibliographystyle{plain}

\end{document}